\documentclass{amsart}

\usepackage[T1]{fontenc}
\usepackage[english]{babel}
\usepackage{graphicx}
\usepackage{verbatim}
\usepackage{tabularx}
\usepackage{color}
\usepackage{amssymb, amsthm, amsmath}
\usepackage{setspace}
\usepackage{scrhack}
\usepackage{caption}
\usepackage{floatrow}
\usepackage{listings}
\usepackage{dsfont}
\usepackage{mathtools}
\usepackage{relsize}
\usepackage{esdiff}
\usepackage{etoolbox}
\usepackage{comment}
\usepackage{accents}
\usepackage[super]{nth}
\usepackage[export]{adjustbox}
\usepackage{scalerel}
\usepackage{url}
\usepackage{enumitem}
\usepackage{calc}
\usepackage{xcolor}
\usepackage{wasysym}
\usepackage{pgfplots}
\pgfplotsset{compat=newest}
\usepackage{tikz}
\usetikzlibrary{intersections, decorations.markings, arrows.meta, pgfplots.fillbetween}

\definecolor{ColMaWi}{rgb}{0.3372,0.6667,0.1098}

\newtheorem{theorem}{Theorem}[section]
\newtheorem{lemma}[theorem]{Lemma}
\newtheorem{proposition}[theorem]{Proposition}
\theoremstyle{definition}
\newtheorem{definition}[theorem]{Definition}
\newtheorem{example}[theorem]{Example}

\newtheorem{corollary}[theorem]{Corollary}
\newtheorem{conj}[theorem]{Conjecture}

\numberwithin{equation}{section}

\makeatletter
\newcommand{\leqnomode}{\tagsleft@true}
\newcommand{\reqnomode}{\tagsleft@false}
\makeatother

\DeclareMathOperator{\argg}{arg}
\DeclareMathOperator{\Int}{Int}
\DeclareMathOperator{\Ext}{Ext}

\newcommand{\R}[1][]{\ensuremath{\mathbb{R}^#1}}
\newcommand{\C}{\ensuremath{\mathbb{C}}}
\newcommand{\Hol}{\ensuremath{\mathcal{O}}}
\newcommand{\N}{\ensuremath{\mathbb{N}}}

\renewcommand{\bf}{\bfseries}

\begin{document}

\reqnomode

\noindent                                             
\begin{picture}(150,36)                               
\put(5,20){\tiny{Submitted to}}                       
\put(5,7){\textbf{Topology Proceedings}}              
\put(0,0){\framebox(140,34)}                          
\put(2,2){\framebox(136,30)}                          
\end{picture}                                         

\vspace{0.5in}

\title[Basins and elliptic sectors in holomorphic flows]
{Basins of Equilibria and geometry of Global Sectors in Holomorphic Flows}

\author{Nicolas Kainz}
\address{Institute of Numerical Mathematics, Ulm University, Germany}
\email{nicolas.kainz@uni-ulm.de}

\author{Dirk Lebiedz}
\address{Institute of Numerical Mathematics, Ulm University, Germany}
\email{dirk.lebiedz@uni-ulm.de}



\thanks{\textit{Acknowledgement:} We thank Jörn Dietrich for mathematical discussions and collaborative solution-finding. We thank the referee for helpful comments, suggestions, and additional literature references.}

\begin{abstract}
In this follow-up paper, we investigate the global geometry and topology of dynamical systems $\dot{x} = F(x)$ with entire vector field $F$, building on and constructively extending the local structure of simple and higher-order equilibria. We provide a step-by-step analysis to reveal topological properties of the basins of centers, nodes, and foci, while excluding isolated equilibria at the boundaries of the latter two. We propose a definition of global elliptic sectors and introduce the concept of sector-forming orbits based on the geometry within a finite elliptic decomposition of multiple equilibria. Finally, we characterize the structure of heteroclinic regions connecting two equilibria.
\end{abstract}

\maketitle

\section{\bf Introduction}
In our previous work \cite{kainz2024local}, we examined the local structure of simple and higher-order equilibria in holomorphic vector fields, i.e. in \textit{holomorphic dynamical systems} of the form
\begin{align}\label{eq:planarODE}
    \dot{x}=\frac{\mathrm{d}x}{\mathrm{d}t}=F(x),\quad x\in\C{},\;t\in\R{}
\end{align}
with $F\in\Hol(\C{})$. Building on these local insights, this follow-up paper shifts the perspective to the global phase portrait of entire vector fields and develops an understanding of how local structures influence and shape the global geometry of \eqref{eq:planarODE}. Throughout this paper, we use the terminology and notion of orbits and flows in the context of dynamical systems and the Jordan curve Theorem from \cite[Chapter 2]{kainz2024local}. In particular, if $\Gamma\subset\C{}$ is a closed Jordan curve, we denote the two connecting components resulting from the Jordan curve Theorem by $\Int(\Gamma)$ (the bounded interior of $\Gamma$) and $\Ext(\Gamma)$ (the unbounded exterior of $\Gamma$).

As examined in \cite{kainz2024local}, each equilibrium of \eqref{eq:planarODE} falls into one of the following categories:
\begin{itemize}
    \item[(i)] A \textbf{center} (simple equilibrium), where all orbits in a neighborhood are closed periodic orbits surrounding the center.
    \item[(ii)] A \textbf{focus} or \textbf{node} (simple equilibrium), which is attracting or repelling such that all nearby orbits tend to the equilibrium either in positive (attracting) or negative (repelling) time.
    \item[(iii)] A multiple equilibrium of order $m\in\mathbb{N}\setminus\{1\}$, possessing a \textbf{finite elliptic decomposition} of order $2m-2$. This geometric structure is defined in \cite[Definition 4.1]{kainz2024local} and illustrated in Section 3.1.
\end{itemize}

With this in mind, we focus on the following two central questions to reveal the relation between local and global dynamics:

\begin{itemize}
    \item[(A)] How can we globally define the "region of influence" of an equilibrium based on its known local structure, and what are the topological properties of the resulting set?  
    \item[(B)] How can these "regions of influence" and their geometric structure be embedded into the existing theory of canonical (strip, annular, radial, and spiral) regions in general smooth planar dynamical systems, cf. \cite{neumann1975classification,markus1954global,markus1954globalintegrals,neumann1976global}?
\end{itemize}

For each of the three cases (i)--(iii) described above, we provide a step-by-step analysis to address questions (A) and (B). In addition, we emphasize the significance and specific influence of the holomorphy of the vector field in \eqref{eq:planarODE} on the possible geometric configurations of the global phase space. This becomes apparent both in the results stated in theorems and in the arguments within proofs. In particular, the central tool is \cite[Corollary 5.1]{kainz2024local}, which can be formulated in the case of entire vector fields as follows:

\begin{theorem}[Poincaré-Bendixson for Holomorphic Flows]\label{thm:central_tool}
    Let $F\in\Hol{}(\C)$, $F\not\equiv 0$, be an entire function and $K\subset\C$ compact with $\xi\in K$ such that $\Gamma_{+(-)}(\xi)\subset K$. Then either $\Gamma(\xi)$ is a periodic orbit with exactly one equilibrium, a center, in its interior, or the positive (negative) limit set of $\Gamma(\xi)$ consists of exactly one equilibrium. Additionally, if $\Gamma(\xi)$ is periodic, the interior of $\Gamma(\xi)$ (except for the center) is filled entirely with periodic orbits all having the center in its interior.
\end{theorem}

The analysis of questions (A) and (B) also sets the stage for studying holomorphic and meromorphic flows of vector fields derived from complex functions pertinent to number theory, such as the Riemann $\zeta$- and $\xi$-function \cite{broughan2005xi,broughan2004zeta,poppe2024sensitivities,lebiedz2020hamiltonian,schleich2018newtonxi} and their polynomial approximations. Some ideas are motivated by previous works, in which we analyzed normally attracting or repelling (slow) invariant manifolds of equilibria within the context of multiscale dynamical systems \cite{heiter2018,Lebiedz2005b,Lebiedz2006a,Lebiedz2011a,Lebiedz2013a,Lebiedz2016,Lebiedz2022}. In the following, we provide an outline of our paper.

We start with a short chapter about the center basin, also called period annulus. We use the well-known theory about closed orbits examined in \cite{andronov1973qualitativetheory,markus1954global,neumann1976global} in order to provide complete proofs of several topological properties that are already stated in \cite[Theorem 4.1]{broughan2003structure} and \cite[Theorem 3.3]{broughan2003holomorphic}.

Subsequently, as the main focus of this paper, we examine the topological structure of equilibria of order $m\ge 2$. Since a precise definition of a \textit{global} elliptic sector is lacking in the literature, cf. \cite{andronov1973qualitativetheory,broughan2003structure,brickman1977conformal,garijo2006local,dumortier2006qualitative,perko2013differential}, we introduce the concept of sector-forming orbits to provide an appropriate global characterization. We show that the intuitive notion of the geometric structure of a global elliptic sector can indeed be made rigorous. We conclude this section by showing that, in the sense of Markus and Neumann \cite{neumann1975classification,markus1954global}, it naturally forms a strip canonical region, thus addressing question (B).


Building on these results, we analyze the topology of the basin of nodes and foci. The authors of \cite[Theorem 4.3]{broughan2003structure} explicitly do not exclude the scenario of an isolated equilibrium on the boundary of the basin. 
We address this aspect in detail in Proposition \ref{prop:nf_no_isolated_equilibria} with a rigorous proof in the appendix.

Finally, we define and investigate heteroclinic regions as one geometric possibility between two adjacent elliptic sectors, observing topological properties similar to the preceding cases. This leads to the exclusion of the existence of a cycle in the phase space formed by equilibria and heteroclinic orbits connecting these equilibria, cf.~Corollary \ref{cor:hetreg_no_cycle_with_heteroclinic_regions}. We conclude this paper with an example that illustrates the occurrence of several elliptic sectors and heteroclinic regions together with their separatrices in the phase space.

Throughout the paper, we selectively focus on explicit constructions that extend local structures to global ones, without a priori resorting to potentially known or conjectured global results, because we consider constructive proofs to offer a fruitful setting for developing general approaches to questions of type (A) and (B). Nevertheless, we aim to contextualize our framework within previous research from different but related perspectives.

\section{\bf Basin of centers}
Throughout this chapter, we consider an entire function $F\in\Hol{}(\C{})$, $F\not\equiv0$, and a center $a\in\C{}$ of \eqref{eq:planarODE}.
\begin{definition}\label{def:center}
    We define the \textit{center basin} or \textit{period annulus} $\mathcal{V}$ of $F$ in $a$ as
    \begin{align*}
        \mathcal{V}:=\{a\}\cup\left\{x\in\C:\Gamma(x)\text{ is periodic with }a\in\Int(\Gamma)\right\}.
    \end{align*}
    The boundary of $a$ is defined as the boundary of $\mathcal{V}$.
\end{definition}
A well-established theory exists regarding the local and global geometry of the period annulus $\mathcal{V}$. Markus and Neumann, cf. \cite{neumann1975classification, markus1954global, neumann1976global}, provide a thorough analysis of annular (parallel) canonical regions. In the following, we highlight which of the known properties extend to the holomorphic setting and examine how holomorphy further refines both the global geometry and the proof techniques. This approach leads to complete proofs of the topological results stated in \cite[Theorem 4.1]{broughan2003structure} and \cite[Theorem 3.3]{broughan2003holomorphic}.

\begin{proposition}\label{prop:center_invariant}
    $\mathring{\mathcal{V}}$ is nonempty. $\mathcal{V}$ and $\partial\mathcal{V}$ are flow-invariant. It holds
    \begin{align}\label{eq:center_invariant}
		\mathcal{V}=\bigcup_{\mathclap{x\in\mathcal{V}}} \overline{\Int(\Gamma(x))}\text{.}
    \end{align}
\end{proposition}
\begin{proof}
    Definition 3.1 in \cite{kainz2024local} ensures $a\in\mathring{\mathcal{V}}\not=\emptyset$. The flow-invariance of $\mathcal{V}$ and $\partial\mathcal{V}=\overline{\mathcal{V}}\cap\overline{\C\setminus\mathcal{V}}$ follow from \cite[Lemma 6.4]{teschl2012ordinary}. By using the holomorphy of $F$, we can apply Theorem \ref{thm:central_tool} to conclude that all orbits in $\Int(\Gamma(x))$, with $x\in\mathcal{V}$, are periodic with $a$ in the interior, which shows \eqref{eq:center_invariant}.
\end{proof}
This Proposition shows that $\mathcal{V}$ is indeed a parallel annular region, cf. \cite[p. 90]{neumann1976global}. In general, equation \eqref{eq:center_invariant} does not hold for arbitrary planar dynamical systems. In fact, there are examples where one or more limit cycles appear within closed orbits, with spiraling behavior in a neighborhood of the limit cycle. However, holomorphy rules out the existence of limit cycles and similar phenomena within periodic orbits, cf. \cite[Theorem 3.2]{broughan2003holomorphic}.

The following result demonstrates that, once again due to holomorphy of $F$, no equilibria can lie on the boundary of $a$.
\begin{proposition}\label{prop:center_no_Equilibria_on_boundary}
    It holds $\partial\mathcal{V}\cap F^{-1}(\{0\})=\emptyset$.
\end{proposition}
\begin{proof}
    The case $\partial\mathcal{V}=\emptyset$ is trivial, so let $x\in\partial\mathcal{V}$ and assume this is an equilibrium. For all $\varepsilon>0$ there exists $y\in\mathcal{B}_\varepsilon(x)\cap\mathcal{V}$ such that $\Gamma(y)$ neither has only $x$ in its limit sets nor is a periodic orbit with $x$ in its interior. Note that $x\not=a$, since $a\in\mathring{\mathcal{V}}$. Hence, $x$ cannot be a center, focus or node. By using the holomorphy of $F$, we can apply \cite[Theorem 4.4]{kainz2024local} to conclude that $x$ has a FED. In particular, by \cite[Definition 4.1]{kainz2024local}, there exists $r>0$ such that $x\in w_+(\Gamma(z))\cup w_-(\Gamma(z))$ for all $z\in\mathcal{B}_r(x)$. But the orbits in $\mathcal{B}_r(x)\cap\mathcal{V}\not=\emptyset$ do not tend to $x$, contradicting the property c) (ii) in \cite[Definition 4.1]{kainz2024local}. Thus, $x$ cannot be an equilibrium.
\end{proof}
\begin{theorem}\label{thm:center_open}
    $\mathcal{V}$ is open and unbounded. All orbits on $\partial\mathcal{V}$ are unbounded.
\end{theorem}
\begin{proof}
    Since $\mathcal{V}$ consists of closed periodic orbits, we can use the already known theory about annular regions of analytic planar dynamical systems. In \cite[§12.1, p. 203]{andronov1973qualitativetheory}, it is shown that $\partial\mathcal{V}$ consists of equilibria as well as separatrices of saddle points. However, the holomorphy of $F$ prevents the existence of saddles, cf. \cite[Theorem 3.2]{kainz2024local}. Hence, by Proposition \ref{prop:center_no_Equilibria_on_boundary}, $\partial\mathcal{V}$ consists only of unbounded separatrices tending to infinity in both time directions. In particular, each of these separatrices connects two saddles at Infinity, cf. \cite{kotus1982global}, and is in particular not a periodic orbit, i.e. $\mathcal{V}\cap\partial\mathcal{V}=\emptyset$.
\end{proof}
We have thus verified that $\mathcal{V}$ is a canonical region whose boundary consists of separatrices, cf. \cite[Theorem 5.2]{markus1954global}. Furthermore, there exists a global smooth transversal in $\mathcal{V}\setminus\{a\}$, that is, an open smooth arc crossing each closed orbit in $\mathcal{V}$ exactly once. The definition of annular canonical regions leads immediately to the following result. However, we present a alternative short proof by applying a version of the Seifert–Van Kampen theorem, combined with the results obtained above.
\begin{theorem}\label{thm:center_simply_connected}
    $\mathcal{V}$ is connected, path-connected and simply connected.
\end{theorem}
\begin{proof}
    The interior of any periodic orbit is clearly connected. Hence, by Proposition \ref{prop:center_invariant}, $\mathcal{V}$ can be interpreted as union of connected sets, all having the point $a$ in common. Thus, also the union must be connected, cf. \cite[Theorem 23.3]{munkres2000topology}. By \cite[Proposition 12.25]{sutherland2009introduction}, $\mathcal{V}$ is even path-connected. By openness of $\mathcal{V}$ we get
    \begin{align*}
		\mathcal{V}=\bigcup_{\mathclap{x\in\mathcal{V}}} \Int(\Gamma(x))\text{.}
    \end{align*}
    All the sets $\Int(\Gamma(x))$, with $x\in\mathcal{V}$, are open and simply connected. Hence, by applying the Seifert-van Kampen theorem for infinite open covers, cf. \cite[Theorem 1.20]{hatcher2005algebraic}, we conclude that also $\mathcal{V}$ is simply connected.
\end{proof}
We conclude this chapter with a lemma that will play an important role at several points in what follows.
\begin{lemma}\label{lem:no_center_on_boundary_of_invariant_sets}
    Let $A\subset\C$ be a flow-invariant set that admits no closed orbits. Then $\mathcal{V}\cap\partial A=\emptyset$.
\end{lemma}
\begin{proof}
    If the set $\mathcal{V}\cap\partial A$ was not empty, it would contradict the openness of $\mathcal{V}$, cf. Theorem \ref{thm:center_open}.
\end{proof}

\section{\bf Global elliptic sectors}
In this chapter, we consider an entire function $F\in\Hol{}(\C)$, $F\not\equiv0$, and an equilibrium $a\in\C$ of \eqref{eq:planarODE} of order $m\in\N\setminus\{1\}$. Our aim is to define and understand the global structure of the "areas of influence" associated with such an equilibrium. We begin by providing a brief overview of earlier research on this topic.

\subsection{Survey of related research}

In \cite[§17]{andronov1973qualitativetheory} and \cite[Chapter 1.5]{dumortier2006qualitative}, the authors describe and analyze several types of curvilinear sectors in general planar dynamical systems, providing detailed proofs of their geometric properties. A global analysis of planar dynamical systems is given in \cite{neumann1975classification, markus1954global, markus1954globalintegrals, neumann1976global}, where canonical regions and global transversals are classified. The structure of multiple equilibria is also discussed there. However, these works do not specifically address the special case of holomorphic vector fields in \eqref{eq:planarODE}.

In \cite{brickman1977conformal}, the authors classify holomorphic flows up to conformal equivalence in a neighborhood of an equilibrium, providing representatives for each equivalence class. This yields several model flows together with their corresponding local phase portraits. A related investigation is carried out in \cite{garijo2004normal}, where the authors describe various normal forms depending on the order and residue of the equilibrium. For example, the local structure of an equilibrium $a$ of order $m\in\N\setminus\{1\}$ is conformally equivalent to \eqref{eq:planarODE} with
\begin{align*}
    F(x):=x^n-cx^{2n-1}\text{,}
\end{align*}
where $c\in\C{}$ is the residuum of $\frac{1}{F}$ at $a$. However, all these analyses concern the local phase portrait.

We have likewise analyzed the local structure of multiple equilibria in our recent paper \cite{kainz2024local}, where we established the existence of a finite elliptic decomposition (FED) of order $2m-2$ in the multiple equilibrium $a$, cf. \cite[Proposition 4.3]{kainz2024local} and \cite[Theorem 4.4]{kainz2024local}. In particular, we showed that each local elliptic sector of this decomposition has adjacent definite directions given by
\begin{align*}
    \mathcal{E}(F,m)=\left\{\frac{\ell\pi-\arg(F^{(m)}(a))}{m-1}\mod 2\pi:\ell\in\mathbb{Z}\right\}\subset[0,2\pi)\text{.}
\end{align*}
The geometry of a local elliptic sector is defined in \cite[Definition 4.1 c)]{kainz2024local} and illustrated in Figure \ref{fig:local_elliptic_sector}. A finite elliptic decomposition can be obtained by cyclically copying the geometry in Figure \ref{fig:local_elliptic_sector} around the equilibrium. In the proof of \cite[Theorem 4.4]{kainz2024local}, we established step by step the various geometric objects required for a finite elliptic decomposition without invoking the theory of topological/conformal equivalence.

\begin{figure}[ht]
    \begin{center}
        \begin{tikzpicture}
             \begin{axis}[axis lines=none, ticks=none, xmin=-0.2, xmax=2, ymin=-0.36, ymax=1.9, scale = 1]
                \pgfmathsetmacro{\L}{0.75};
                \pgfmathsetmacro{\LGray}{0.85*\L}
                \pgfmathsetmacro{\LArrows}{0.6*\L}
                \pgfmathsetmacro{\P}{1.3};
                
                \addplot[draw=none, fill=blue!30, opacity=0.9] table[col sep=comma] {homoclinicOrbit.csv};
                
                \addplot[color=ColMaWi, name path = homoclinicOrbit, postaction={decorate}, decoration={markings, mark= at position 0.25 with {\arrow{Classical TikZ Rightarrow}}, mark= at position 0.55 with {\arrow{Classical TikZ Rightarrow}}, mark= at position 0.8 with {\arrow{Classical TikZ Rightarrow}}}, line width=\L pt, line join = round, line cap = round] table[col sep=comma] {homoclinicOrbit.csv};
                \node at (0.9,0.7) [right] {\textcolor{ColMaWi}{$\Xi$}};
                \fill (0,0) circle (2pt) node[left,yshift=-1.4mm] {$a$};
                \path[name path=path1] (0.890285394129103,0.408775954806926) .. controls (1.25,0.1) .. (1.2,-0.2);
                \path[name path=path2] (0.530612668121931,0.747386591257611) .. controls (0.25,1.35) .. (-0.3,1.3);
                \draw[name path=sep1, postaction={decorate}, decoration={markings, mark= at position 0.3 with {\arrow{Classical TikZ Rightarrow}}, mark= at position 0.85 with {\arrow{Classical TikZ Rightarrow}}}, line width=\L pt, line join = round, line cap = round] (0,0) .. controls (1.2,-0.2) .. (1.9,0.05);
                \draw[name path=sep2, postaction={decorate}, decoration={markings, mark= at position 0.3 with {\arrowreversed{Classical TikZ Rightarrow}}, mark= at position 0.87 with {\arrowreversed{Classical TikZ Rightarrow}}}, line width=\L pt, line join = round, line cap = round] (0,0) .. controls (-0.15,0.9) .. (0.15,1.8);
                \fill[name intersections={of=sep1 and path1, by={in1}}] (in1) circle (\P pt) node[below,yshift=-0.3mm] {$p_1$};
                \fill[name intersections={of=sep2 and path2, by={in2}}] (in2) circle (\P pt) node[left] {$p_2$};
                
                \input{homoclinicOrbitE1Lambda1}; 
                \addplot [blue!30, opacity=0.9] fill between [of=top1 and sep1];
                \path[fill=white] (in1) --  (1.3,1) -- (1.9,0.05) -- (1.6,-0.25) -- cycle;
                
                \input{homoclinicOrbitE2Lambda2}; 
                \addplot [blue!30, opacity=0.9] fill between [of=top2 and sep2];
                
                \draw[red,line width=\LGray pt, line join = round, line cap = round] (0.890285394129103,0.408775954806926) .. controls (1.2,0.2) .. (in1);
                \draw[red,line width=\LGray pt, line join = round, line cap = round] (0.530612668121931,0.747386591257611) .. controls (0.35,1.2) .. (in2);
                \draw[-{Classical TikZ Rightarrow}, line width=\LArrows pt, line join = round, line cap = round] (0.89,0.12) .. controls (1.14,0.32) .. (1.18,0.6);
                \draw[-{Classical TikZ Rightarrow}, line width=\LArrows pt, line join = round, line cap = round] (0.98,-0.04) .. controls (1.4,0) .. (1.6,0.22);
                \draw[-{Classical TikZ Rightarrow}, line width=\LArrows pt, line join = round, line cap = round] (0.4,1.5) .. controls (0.17,1.33) .. (0.07,0.95);
                \draw[-{Classical TikZ Rightarrow}, line width=\LArrows pt, line join = round, line cap = round] (0.8,0.98) .. controls (0.51,1.02) .. (0.24,0.74);
                \fill (0.890285394129103,0.408775954806926) circle (\P pt) node[left,yshift=1mm,xshift=0.3mm] {$E_1$};
                \fill (0.530612668121931,0.747386591257611) circle (\P pt) node[below,xshift=0.8mm] {$E_2$};
                \node at (1.15,0.22) [right] {\textcolor{red}{$\Lambda_1$}};
                \node at (0.33,1.19) [right] {\textcolor{red}{$\Lambda_2$}};
                \node at (1.88,0.04) [above] {$\Gamma_1$};
                \node at (0.14,1.79) [right] {$\Gamma_2$};
                \fill (in1) circle (\P pt);
                \fill (in2) circle (\P pt);
                
                \draw[name path=sep1, postaction={decorate}, decoration={markings, mark= at position 0.3 with {\arrow{Classical TikZ Rightarrow}}, mark= at position 0.85 with {\arrow{Classical TikZ Rightarrow}}}, line width=\L pt, line join = round, line cap = round] (0,0) .. controls (1.2,-0.2) .. (1.9,0.05);
                \draw[name path=sep2, postaction={decorate}, decoration={markings, mark= at position 0.3 with {\arrowreversed{Classical TikZ Rightarrow}}, mark= at position 0.87 with {\arrowreversed{Classical TikZ Rightarrow}}}, line width=\L pt, line join = round, line cap = round] (0,0) .. controls (-0.15,0.9) .. (0.15,1.8);
            \end{axis}
        \end{tikzpicture}
    \end{center}
    \caption{Geometrical objects of a local elliptic sector (blue) in a multiple equilibrium $a$. $\Gamma_1$ and $\Gamma_2$ (black) are the characteristic orbits. $\Lambda_1$ and $\Lambda_2$ (red) are the transversals or paths without contact. $\Xi=\Gamma(E_1)$ (green) is the homoclinic orbit. The arrows indicate the direction of the vector field.}
    \label{fig:local_elliptic_sector}
\end{figure}

With this in mind, we now extend the analysis of the local structure to a global perspective by introducing the notion of a global elliptic sector. We provide detailed complete proofs of the topological results given in \cite[Theorem 4.2]{broughan2003structure}. Furthermore, we embed this global elliptic sector as a canonical region into the global theory of Markus and Neumann, cf. \cite{markus1954global, markus1954globalintegrals, neumann1976global}.

\subsection{The geometry of global elliptic sectors}
In this chapter, we provide a step-by-step analysis of the global elliptic sector to obtain a complete topological and geometrical characterization. We introduce the notion of the global elliptic sector in an intuitive manner by "blowing up" the local structure in an appropriate way.
\begin{definition}\label{def:globalSector}
    \leavevmode
    \begin{itemize}
        \item[(i)] Let $\Xi\subset\C\setminus\{a\}$ be a homoclinic orbit in $a$, i.e. $w_+(\Gamma)=w_-(\Gamma)=\{a\}$. $\Xi$ is a \textit{sector-forming orbit} in $a$, if for all $z\in\Int(\Xi\cup\{a\})$ the orbit $\Gamma(z)$ is also homoclinic in $a$.\footnote{A construction of a parameterization for the closed Jordan curve $\Xi\cup\{a\}$ with compact time interval can be found in \cite[Remark 4.2]{kainz2024local}. }
        \item[(ii)] Let $\Xi$ be a sector-forming orbit in $a$. The \textit{global elliptic sector} $\mathcal{S}(\Xi)$ of $F$ in $a$ with respect to $\Xi$ is
        \begin{align*}
            \mathcal{S}(\Xi):=\underbrace{\Xi\cup\Int(\Xi\cup\{a\})}_{=\overline{\Int(\Xi\cup\{a\})}\setminus\{a\}}\cup\,\mathcal{S}^\prime(\Xi)
        \end{align*}
        with
        \begin{align*}
             \mathcal{S}^\prime(\Xi):=\big\{x\in\C:\,&\Gamma(x)\text{ is homoclinic in }a,\\
             &\Xi\subset\Int(\Gamma(x)\cup\{a\})\big\}\text{.}
        \end{align*}
        \item[(iii)] Let $\Gamma_1,\Gamma_2$, be two homoclinic orbits in $a$ and $\Gamma_1\not=\Gamma_2$. $\Gamma_1$ and $\Gamma_2$ are called \textit{nested} if either $\Gamma_1\subset\Int(\Gamma_2\cup\{a\})$ or $\Gamma_2\subset\Int(\Gamma_1\cup\{a\})$. $\Gamma_1$ and $\Gamma_2$ are called \textit{mutually exterior} if they are not nested, i.e. $\Gamma_1\subset\Ext(\Gamma_2\cup\{a\})$ and $\Gamma_2\subset\Ext(\Gamma_1\cup\{a\})$.\footnote{This terminology is in the sense of \cite[\S 17, 3.]{andronov1973qualitativetheory}.}
    \end{itemize}
\end{definition}
Apriori the global elliptic sector in this constructive Definition depends on the choice of the sector-forming orbit. We will later see the motivation for this type of construction and demonstrate that the global elliptic sector is in an appropriate sense independent of the choice of a particular sector-forming orbit -- as can be intuitively deduced from the idea of the structure of an equilibrium of order $m\ge2$, cf. Figures 2 and 3 in \cite{kainz2024local}.
\begin{lemma}\label{lem:existence_of_sector_forming_orbit}
    There exists a sector-forming orbit $\Xi$ in $a$.
\end{lemma}
\begin{proof}
    Since there exists a FED in $a$, we can choose two adjacent definite directions $\theta_+,\theta_-\in\mathcal{E}(F,m)$ and the orbit $\Xi:=\Gamma(E_1)$, where $E_1$ is the point given by \cite[Definition 4.1 c)]{kainz2024local}, cf. Figure \ref{fig:local_elliptic_sector}. Then $w_+(\Xi)=w_-(\Xi)=\{a\}$, $\arg(\Phi(t,E_1)-a)=\theta_{+}$ for $t\to\infty$ and $\arg(\Phi(t,E_1)-a)=\theta_{-}$ for $t\to-\infty$, i.e. $\Xi$ is homoclinic and tends to $a$ in the definite direction $\theta_\pm$ for $t\to\pm\infty$. This orbit is part of a local elliptic sector $S$ between $\theta_+$ and $\theta_-$. In addition, all orbits in $\overline{\Int(\Xi\cup\{a\})}\setminus\{a\}\subset S$ are homoclinic in $a$, cf. \cite[Definition 4.1 c)]{kainz2024local}.
\end{proof}
The above Lemma ensures the existence of a sector-forming orbit between two arbitrarily chosen adjacent definite directions. In the following Propositions and Theorems we present various topological properties of the global elliptic sector.
\begin{proposition}\label{prop:globalSector_invariant}
    Let $\mathcal{S}:=\mathcal{S}(\Xi)$ be a global elliptic sector generated by the sector-forming orbit $\Xi$. $\mathring{\mathcal{S}}$ is nonempty. $\mathcal{S}$ and $\partial\mathcal{S}$ are flow-invariant. It holds
    \begin{align*}
		\mathcal{S}=\bigcup_{\mathclap{x\in\mathcal{S}}}\Gamma(x)\text{.}
    \end{align*}
\end{proposition}
\begin{proof}
    By choosing a point in $\Int(\Xi\cup\{a\})$, we have $\mathring{\mathcal{S}}\not=\emptyset$. The rest of the proof is analogous to the proof of Proposition \ref{prop:center_invariant}.
\end{proof}
\begin{proposition}\label{prop:globalSector_Equilibria_on_boundary}
    Let $\mathcal{S}:=\mathcal{S}(\Xi)$ be a global elliptic sector generated by the sector-forming orbit $\Xi$. It holds $\partial\mathcal{S}\cap F^{-1}(\{0\})=\{a\}$.
\end{proposition}
\begin{proof}
    Let $x\in\partial\mathcal{S}\setminus\{a\}$. Suppose $x$ is an equilibrium. Since the orbits in $\mathcal{S}$ all tend to $a\not=x$ in both time directions and are not periodic, $x$ cannot be a center, focus or node. Hence, as in the proof of Proposition \ref{prop:center_no_Equilibria_on_boundary}, $x$ has a FED and there exists $r>0$ such that $x\in w_+(\Gamma(z))\cup w_-(\Gamma(z))$ for all $z\in\mathcal{B}_r(x)$ due to the holomorphy of $F$, cf. \cite[Theorem 4.4]{kainz2024local}. Since $\mathcal{B}_r(x)\cap\mathcal{S}\not=\emptyset$, we get the contradiction $x=a$. Thus, $x$ cannot be an equilibrium. This proves $\partial\mathcal{S}\cap F^{-1}(\{0\})\subset\{a\}$. Moreover, $\Xi\subset\mathcal{S}$ tends to $a$, i.e. $a\in\overline{\mathcal{S}}$. Since $a\not\in\mathcal{S}$, we conclude $a\in\overline{\mathcal{S}}\setminus\mathcal{S}\subset\partial\mathcal{S}$ and thus $\partial\mathcal{S}\cap F^{-1}(\{0\})=\{a\}$.
\end{proof}
\begin{lemma}\label{lemma:globalSector_nested_orbits}
    Let $\Gamma$ be a homoclinic orbit in $a$ and $x\in\Gamma$. Let $\varepsilon>0$. Then there exists $\delta>0$ such that all orbits $\hat{\Gamma}$ through $\mathcal{B}_\delta(x)$ satisfy the following properties:
    \begin{itemize}
        \item[(i)] $\hat{\Gamma}$ is homoclinic in $a$.
        \item[(ii)] There exists a set $H\subset\mathcal{U}_\varepsilon(\Gamma)$ such that $\partial H=\Gamma\cup\hat{\Gamma}\cup\{a\}$.\footnote{The set $H$ looks similar to a crescent moon \leftmoon.}
        \item[(iii)] If $\Gamma\not=\hat{\Gamma}$, then $\Gamma$ and $\hat{\Gamma}$ are nested.
    \end{itemize}
\end{lemma}
\begin{proof}
    Due to the holomorphy of $F$, we can apply \cite[Proposition 4.3 (ii)]{kainz2024local} to conclude the existence of two angles $\theta_+,\theta_-\in\mathcal{E}(F,m)$ such that $\Gamma$ tends to $a$ in the definite directions $\theta_+$ and $\theta_-$. By \cite[Proposition 4.3 (iii)]{kainz2024local} there exist $r,\Theta_\pm>0$ such that all orbits through the open set $\{z\in\C{}:|z-a|<r,|\arg(z-a)-\theta_\pm|<\Theta_\pm\}=:A_\pm$ tend to $a$ in the definite direction $\theta_\pm$ for $t\to\pm\infty$, i.e. are in particular homoclinic. Moreover, there exists $\tau>0$ such that $\Phi(\pm\tau,x)\in A_\pm$. Choose $\rho>0$ small enough such that $\mathcal{B}_{\rho}(\Phi(\pm\tau,x))\subset A_\pm$. By continuity of the flow, cf. \cite[Chapter 2.4, Theorem 4]{perko2013differential}, there exists a $\delta>0$ such that $|\Phi(\pm\tau,x)-\Phi(\pm\tau,y)|<\rho$ for all $y\in\mathcal{B}_{\delta}(x)$. This shows that all orbits $\hat{\Gamma}$ passing the circle $\mathcal{B}_{\delta}(x)$ are in particular homoclinic in $a$. Furthermore, by \cite[§16.7, Lemma 15]{andronov1973qualitativetheory}, we can reduce $\delta$ such that there exists a region $H\subset\mathcal{U}_\varepsilon(\Gamma)$, whose boundary is $\partial H=\Gamma\cup\hat{\Gamma}\cup\{a\}$. By following the argument in the proof of \cite[§16.10, Lemma 18]{andronov1973qualitativetheory}, we conclude that $\Gamma$ and $\hat{\Gamma}$ are nested if $\Gamma\not=\hat{\Gamma}$. Otherwise $H$ would be unbounded and could not lie within the bounded set $\mathcal{U}_\varepsilon(\Gamma)$.
\end{proof}
\begin{theorem}\label{thm:globalSector_open}
    Let $\mathcal{S}:=\mathcal{S}(\Xi)$ be a global elliptic sector generated by the sector-forming orbit $\Xi$ and $\mathcal{S}^\prime:=\mathcal{S}^\prime(\Xi)$. Then $\mathcal{S}$ is open. In particular, for all $x\in\mathcal{S}$ there exists $y\in\mathcal{S}\cap\Ext(\Gamma(x)\cup\{a\})$ such that $\Gamma(x)\subset\Int(\Gamma(y)\cup\{a\})$.
\end{theorem}
\begin{proof}
    Let $x\in\mathcal{S}$. If $x$ lies in the open set $\Int(\Xi\cup\{a\})$, we are done. If $x\in\Xi$, we can choose $\varepsilon>0$ such that $\mathcal{B}_\varepsilon(x)\cap F^{-1}(\{0\})=\emptyset$ and apply Lemma \ref{lemma:globalSector_nested_orbits}. 
    Hence, it remains to proof the Theorem for the case $x\in\mathcal{S}^\prime$, which is open. We have $\Xi\subset\Int(\Gamma(x)\cup\{a\})$ and choose $\varepsilon>0$ sufficiently small such that $\mathcal{B}_\varepsilon(x)\subset\mathcal{S}^\prime$. By applying Lemma \ref{lemma:globalSector_nested_orbits}, we get $\delta>0$ such that the orbits $\Gamma(x)$ and $\Gamma(y)$, with $y\in\mathcal{B}_\delta(x)$, are nested. The case $y\in\mathcal{B}_\delta(x)\cap\Ext(\Gamma(x)\cup\{a\})$ leads to $\Xi\subset\Int(\Gamma(x)\cup\{a\})\subset\Int(\Gamma(y)\cup\{a\})$. Let $y\in\mathcal{B}_\delta(x)\cap\Int(\Gamma(x)\cup\{a\})$. We have to show $\Xi\subset\Int(\Gamma(y)\cup\{a\})$. If $\Xi\subset\Ext(\Gamma(y)\cup\{a\}$, then $\Xi$ and $\Gamma(y)$ are mutually exterior and $\Xi\subset\mathcal{U}_\varepsilon(\Gamma(x))$. This implies in particular $\text{dist}(\Xi,x)<\varepsilon$, which is a contradiction to our choice of $\varepsilon$.
\end{proof}
The openness of the global elliptic sector now allows us to establish its unboundedness. At this stage, the holomorphy of $F$ plays a crucial role. For general vector fields, a global elliptic sector may in fact be bounded -- particularly when its boundary contains a saddle whose two separatrices both connect to the multiple equilibrium $a$, cf. \cite[Chapter 2]{brickman1977conformal}.
\begin{theorem}\label{thm:globalSector_unbounded}
    Let $\mathcal{S}:=\mathcal{S}(\Xi)$ be a global elliptic sector generated by the sector-forming orbit $\Xi$. All orbits on $\partial\mathcal{S}\setminus\{a\}$ are unbounded. $\mathcal{S}$ is unbounded.
\end{theorem}
\begin{proof}
    Let $x\in\partial\mathcal{S}\setminus\{a\}$ and suppose that $\Gamma(x)\subset\partial\mathcal{S}\setminus\{a\}$ is bounded. The set $K:=\overline{\Gamma(x)}$ is compact and $\Gamma(x)\cup w_\pm(\Gamma(x))\subset K\subset\partial\mathcal{S}$. By using the holomorphy of $F$, we can apply Theorem \ref{thm:central_tool} to conclude that $\Gamma(x)$ either is a periodic orbit with exactly one equilibrium, a center $\tilde{a}$, in $\Int(\Gamma(x))$, or $w_+(\Gamma(x))$ and $w_-(\Gamma(x))$ each consist of exactly one equilibrium. If the first case occurs, by openness of the basin $\mathcal{V}$ in $\tilde{a}$, cf. Theorem \ref{thm:center_open}, there would be a point $y\in\mathcal{S}\cap\mathcal{V}$, $y\not=x$. This is impossible and the second case in Theorem \ref{thm:central_tool} must occur. Hence, by Proposition \ref{prop:globalSector_Equilibria_on_boundary}, $\Gamma(x)$ must be a homoclinic orbit with $w_\pm(\Gamma(x))=\{a\}$. Lemma \ref{lemma:globalSector_nested_orbits} gives us a homoclinic orbit $\hat{\Gamma}$ in $\mathcal{S}$ such that $\Gamma(x)$ and $\hat{\Gamma}$ are nested. However, by openness of $\mathcal{S}$, this is impossible. Thus, $\Gamma(x)$ cannot be homoclinic and must be unbounded. It follows that also $\mathcal{S}$ is unbounded.
\end{proof}
\begin{proposition}\label{prop:globalSector_interior}
    Let $\mathcal{S}:=\mathcal{S}(\Xi)$ be a global elliptic sector generated by the sector-forming orbit $\Xi$. It holds
    \begin{align}\label{eq:globalSector_interior}
	\mathcal{S}=\bigcup_{\mathclap{x\in\mathcal{S}}} \overline{\Int(\Gamma(x)\cup\{a\})}\setminus\{a\} =\bigcup_{\mathclap{x\in\mathcal{S}}}\Int(\Gamma(x)\cup\{a\})\text{.}
    \end{align}
    Moreover, all orbits in $\mathcal{S}$ are nested, i.e. for all $y_1,y_2\in\mathcal{S}$, $\Gamma(y_1)\not=\Gamma(y_2)$, there holds either $\Gamma(y_1)\subset\Int(\Gamma(y_2)\cup\{a\})$ or $\Gamma(y_2)\subset\Int(\Gamma(y_1)\cup\{a\})$.
\end{proposition}
\begin{proof}
    This proof requires several case distinctions and is quite technical. A proof can be found in the appendix.
\end{proof}
\begin{theorem}\label{thm:globalSector_simply_connected}
    Let $\mathcal{S}:=\mathcal{S}(\Xi)$ be a global elliptic sector generated by the sector-forming orbit $\Xi$. $\mathcal{S}$ is connected, path-connected and simply connected.
\end{theorem}
\begin{proof}
    Analogous to the proof of Theorem \ref{thm:center_simply_connected}, this follows directly from \eqref{eq:globalSector_interior} together with the Seifert-van Kampen theorem for infinite open covers, cf. \cite[Theorem 1.20]{hatcher2005algebraic}.
    %
\end{proof}
We now address the independence of the global elliptic sector from a particular choice of the sector-forming orbit ensuring in some sense the well-definedness of the global elliptic sector.
\begin{lemma}\label{lem:globalSector_independent_sector-forming_orbit}
    Let $\mathcal{S}:=\mathcal{S}(\Xi)$ be a global elliptic sector generated by the sector-forming orbit $\Xi$. For all $x\in\mathcal{S}$ the orbit $\Gamma(x)$ is a sector-forming orbit in $a$ and $\mathcal{S}=\mathcal{S}(\Gamma(x))$, i.e. the global elliptic sector $\mathcal{S}$ does not depend on the particular choice of a sector-forming orbit.
\end{lemma}
\begin{proof}
    Let $x\in\mathcal{S}$. By \eqref{eq:globalSector_interior}, all orbits in $\Int(\Gamma(x)\cup\{a\})$ are homoclinic and $\hat{\Xi}:=\Gamma(x)$ is indeed a sector-forming orbit in $a$.
    
    We define $\hat{\mathcal{S}}:=\mathcal{S}(\hat{\Xi})$. If $\hat{\Xi}=\Xi$, then obviously $\mathcal{S}=\hat{\mathcal{S}}$. We have to show $\mathcal{S}\subset\hat{\mathcal{S}}$ as well as $\hat{\mathcal{S}}\subset\mathcal{S}$ for the two cases $\hat{\Xi}\subset\Int(\Xi\cup\{a\})$ and $\Xi\subset\Int(\hat{\Xi}\cup\{a\})$. As these four proofs are all based on similar arguments, we assume w.l.o.g. $\hat{\Xi}\subset\Int(\Xi\cup\{a\})$ and show only $\mathcal{S}\subset\hat{\mathcal{S}}$. Let $y\in\mathcal{S}$. If $y\in\Xi\cup\hat{\Xi}\cup\Int(\hat{\Xi}\cup\{a\})$, then nothing is to show. Thus, assume $y\in\Ext(\hat{\Xi}\cup\{a\})\setminus\Xi$. If $y\in\Ext(\Xi\cup\{a\})$, then we have $\Xi\subset\Int(\Gamma(y)\cup\{a\})$ and thus $\hat{\Xi}\subset\Int(\Xi\cup\{a\})\subset\Int(\Gamma(y)\cup\{a\})$, i.e. $y\in\hat{\mathcal{S}}$. Additionally, if $y\in\Ext(\hat{\Xi}\cup\{a\})\cap\Int(\Xi\cup\{a\})$, then we can apply Proposition \ref{prop:globalSector_interior} to conclude $\hat{\Xi}\subset\Int(\Gamma(y)\cup\{a\})$, i.e. $y\in\hat{\mathcal{S}}$.
\end{proof}

\subsection{Final Corollaries}

Finally, at the end of this chapter, we supplement and enhance our topological study with several important corollaries, which confirm the intuitive idea of the geometrical structure of an equilibrium of order $m\ge2$, cf. \cite[Figure 2, Figure 3]{kainz2024local}.
\begin{corollary}\label{cor:globalSector_final_corrollaries_1}
    \it There exist exactly $2m-2$ distinct\footnote{This means that $\Gamma_1$ and $\Gamma_2$ are mutually exterior, if if they belong to two different global elliptic sectors.} global elliptic sectors in $a$, each located between two adjacent definite directions given by $\mathcal{E}(F,m)$. All homoclinic orbits are sector-forming orbits.
\end{corollary}
\begin{proof}
    By \cite[Theorem 4.4]{kainz2024local}, there is a FED of order $d:=2m-2$ in $a$. We use the geometrical objects $\Gamma$, $\Gamma_i$ and $p_i$, $i\in\{1,\ldots,d\}$, defined by \cite[Definition 4.1 d)]{kainz2024local}, cf Figure \ref{fig:local_elliptic_sector}. We order the definite directions $\mathcal{E}(F,m)=\{\theta_1,\ldots,\theta_d\}$ in such a way that for all $i\in\{1,\ldots,d\}$ the local elliptic sector with characteristic orbits $\Gamma_i$ and $\Gamma_{i+1}$ lies between $\theta_i$ and $\theta_{i+1}$, i.e. $\Gamma_i$ tends to $a$ in the definite direction $\theta_i$. By using \cite[Definition 4.1 c) (iv)]{kainz2024local}, there exist points $E_i\in\Gamma(p_i,p_{i+1})$ such that $\Xi_i:=\Gamma(E_i)$ is a sector-forming orbit with definite directions $\theta_i,\theta_{i+1}$ for all $i\in\{1,\ldots,d\}$. Moreover, since all local sectors in $a$ have pairwise empty intersection except for the characteristic orbits and $a$, these sector-forming orbits are pairwise mutually exterior. This leads to $d$ global elliptic sectors, generated by the $d$ pairwise mutually exterior sector-forming orbits $\Xi_1,\ldots,\Xi_d$. On the one hand, due to the local geometry of $a$ it is clear that each global elliptic sector admits at least one of the orbits $\Xi_1,\ldots,\Xi_d$. On the other hand, since all orbits in a global elliptic sector are nested, it is not possible to find two orbits $\Xi_i,\Xi_j$, $i\not=j$, lying in the same global elliptic sector. Hence, we get indeed exactly $d$ distinct global elliptic sectors in $a$. The rest follows from the geometry described in \cite[Proposition 4.3]{kainz2024local}.
\end{proof}
We have proved Corollary \ref{cor:globalSector_final_corrollaries_1} by using the geometry of the local structure in $a$. This structure is defined by a FED, which is not uniquely determined, i.e. there exist several FEDs in $a$. However, we can now see that the global structure is the unique globalization of this local structure.

In the following, we consider the case $m\ge 3$. In this setting, the global elliptic sector can be defined independently of sector-forming orbits, solely by means of the chosen adjacent definite directions. For this purpose, we require the function $\lambda:\mathcal{E}(F,m)\to\{-1,1\}$,
\begin{align*}
    \lambda(\theta):=\cos(\argg(F^{(m)}(a))+\theta(m-1))
\end{align*}
to determine the time direction (positive or negative) in which the orbits tend to $a$ in the FED, cf. \cite[Proposition 4.3 (iii)]{kainz2024local} as well as the calculation of $\lambda_i$ on page 11 in this reference.
\begin{definition}
    Let $m\ge 3$.
    \begin{itemize}
        \item[(i)] Let $\Gamma\subset\C$ be an orbit and $x\in\Gamma$. We call $\Gamma$ $(\theta_+,\theta_-)$-elliptic in $a$, if there exist $\theta_+,\theta_-\in\mathcal{E}(F,m)$ such that $w_+(\Gamma)=w_-(\Gamma)=\{a\}$, $\arg(\Phi(t,x)-a)=\theta_{+}$ for $t\to\infty$ and $\arg(\Phi(t,x)-a)=\theta_{-}$ for $t\to-\infty$, that is, $\Gamma$ is homoclinic and tends to $a$ in the definite direction $\theta_\pm$ for $t\to\pm\infty$.
        \item[(ii)] Let $\theta_+,\theta_-\in\mathcal{E}(F,m)$ be two adjacent definite directions satisfying $\lambda(\theta_\pm)=\mp 1$.\footnote{If the corresponding local elliptic sector between $\theta_+$ and $\theta_-$ has clockwise (counterclockwise) direction, then $\theta_{-(+)}=\theta_{+(-)}+\frac{\pi}{m-1}$.} We define the set
        \begin{align}\label{eq:globalSector_m_3}
        \mathcal{S}(\theta_+,\theta_-):=\left\{x\in\C{}:\Gamma(x)\text{ is }(\theta_+,\theta_-)\text{-elliptic in }a\right\}\text{.}
    \end{align}
    \end{itemize}
\end{definition}
\newpage
\begin{corollary}
    \it Let $m\ge3$.
    \begin{itemize}
        \item[(i)] Let $\mathcal{S}:=\mathcal{S}(\Xi)$ be a global elliptic sector generated by the sector-forming orbit $\Xi$. Then there exist $\theta_+,\theta_-\in\mathcal{E}(F,m)$ such that
        \begin{align}\label{eq:cor_sector_1}
            \mathcal{S}(\Xi)=\mathcal{S}(\theta_+,\theta_-)\text{.}
        \end{align}
        \item[(ii)] Let $\theta_+,\theta_-\in\mathcal{E}(F,m)$ be two adjacent definite directions satisfying $\lambda(\theta_\pm)=\mp 1$. Then there exists a sector forming orbit $\Xi$ such that
        \begin{align}\label{eq:cor_sector_2}
            \mathcal{S}(\theta_+,\theta_-)=\mathcal{S}(\Xi)\text{.}
        \end{align}
    \end{itemize}
\end{corollary}
\begin{proof}
    The key point underlying this result is that, for $m\ge 3$, there exist more than two definite directions. Since each of the $2m-2$ distinct unbounded global elliptic sectors lies between two adjacent definite directions, there always exists exactly one global elliptic sector $\mathcal{S}_0$ between two adjacent arbitrarily chosen angles $\theta_+$ and $\theta_-$, cf. Corollary \ref{cor:globalSector_final_corrollaries_1}. Moreover, by applying the second statement in Corollary \ref{cor:globalSector_final_corrollaries_1}, all $(\theta_+,\theta_-)$-elliptic orbits must lie in one of these $2m-2$ global elliptic sectors. This shows \eqref{eq:cor_sector_1} and \eqref{eq:cor_sector_2}.
\end{proof}

This second corollary shows that we need to choose a sector-forming orbit only in the case $m=2$. Here we have only two definite directions, which requires the choice of one arbitrary homoclinic orbit to specify one of the two global elliptic sectors. The two sectors cannot be distinguished on the basis of the definite directions alone. If we consider the set on the right hand side of \eqref{eq:globalSector_m_3} in the case $m=2$, then this set is still open and unbounded, but not connected. The set would become connected (but no longer open) only by adding the equilibrium $a$. For this reason, the construction in Definition \ref{def:globalSector} is based on the notion of a sector-forming orbit, which is defined independently of the choice of definite directions.

We observe that the boundaries of two adjacent global elliptic sectors in general do not have to coincide near $a$, as it is the case for the characteristic orbits in a FED. However, if they indeed coincide between all adjacent sectors, the local structure of the FED can even be transferred to the global phase portrait. The question therefore arises as to what the global geometric structure looks like in a region between two adjacent global elliptic sectors having no common boundary near $a$. How many separatrices can tend to a multiple equilibrium (either in positive or negative time)? This question takes up the discussions at the end of \cite[Example 4.7]{kainz2024local} and \cite[Example 4.8]{kainz2024local}. We will discuss these issues in Chapter 5.

Finally, we conclude by addressing the question of how the global elliptic sector fits into the theory of canonical regions.

\begin{corollary}
    \it Let $\mathcal{S}:=\mathcal{S}(\Xi)$ be a global elliptic sector generated by the sector-forming orbit $\Xi$. $\mathcal{S}$ is a canonical region of strip type. There exists a global smooth transversal through $\mathcal{S}$, i.e. an open smooth arc crossing each orbit in $\mathcal{S}$ exactly once.
\end{corollary}
\begin{proof}
    By openness of $\mathcal{S}$, all orbits in $\mathcal{S}$ are part of a strip canonical region $R$. Thus, there are no separatrices in $\mathcal{S}$. Since $\partial\mathcal{S}$ does not have any homoclinic orbits, we have $\partial\mathcal{S}\cap R=\emptyset$, i.e. $R=\mathcal{S}$. The rest follows directly from \cite[Theorem 5.1]{markus1954global} and \cite[Theorem I]{markus1954globalintegrals}.
\end{proof}

\section{\bf Basin of nodes and foci}
In this chapter, we consider an entire function $F\in\Hol{}(\C)$, $F\not\equiv0$, and a node or focus $a\in\C$ of \eqref{eq:planarODE}. We provide detailed proofs of the topological results given in \cite[Theorem 4.3]{broughan2003structure}.
\begin{definition}\label{def:node_focus}
    The \textit{stable (unstable) basin of attraction (repulsion)} $\mathcal{N}$ of $F$ in $a$ is defined as
    \begin{align*}
        \mathcal{N}:=\left\{x\in\C:w_{+(-)}(\Gamma(x))=\{a\}\right\}.
    \end{align*}
    The boundary of $a$ is defined as the boundary of $\mathcal{N}$.
\end{definition}
\begin{proposition}\label{prop:nf_invariant}
    $\mathcal{N}$ and $\partial\mathcal{N}$ are flow-invariant. It holds
    \begin{align*}
		\mathcal{N}= \bigcup_{\mathclap{x\in\mathcal{N}}}\Gamma(x)\text{.}
    \end{align*}
\end{proposition}
\begin{proof}
    This proof is analogous to the proof of Proposition \ref{prop:center_invariant}.
\end{proof}
\begin{theorem}\label{thm:nf_open}
    It holds $a\in\mathring{\mathcal{N}}$. $\mathcal{N}$ is open, connected and path-connected.
\end{theorem}
\begin{proof}
    Assume w.l.o.g. that $a$ is stable. By \cite[Definition 3.1]{kainz2024local}, there exists $\varepsilon_1>0$ such that for all $y\in\mathcal{B}_{\varepsilon_1}(a)$ we have $w_+(\Gamma(y))=\{a\}$. This proves $\mathcal{B}_{\varepsilon_1}(a)\subset\mathcal{N}$ and $a\in\mathring{\mathcal{N}}$. Let $x\in\mathcal{N}\setminus\{a\}$ be arbitrary. Then there exists a $\tau>0$ with $\xi:=\Phi(\tau,x)\in\mathcal{B}_{\varepsilon_1}(a)$. Choose $\varepsilon_2:=\frac{1}{2}\min\{|\xi-a|,\varepsilon_1-|\xi-a|\}>0$ such that $\mathcal{B}_{\varepsilon_2}(\xi)\subset\mathcal{B}_{\varepsilon_1}(a)$. By continuity of the flow, cf. \cite[Chapter 2.4, Theorem 4]{perko2013differential}, there exists a $\delta>0$ such that $|\Phi(\tau,z)-\xi|<\varepsilon_2$ for all $z\in\mathcal{B}_\delta(x)$. Thus $w_+(\Gamma(z))=\{a\}$ for all $z\in\mathcal{B}_\delta(x)$ and $\mathcal{N}$ is open.
    
    For $x\in\mathcal{N}$ set $J_x:=\Gamma(x)\cup\{a\}$ and define the cover $C:=\{J_x:x\in\mathcal{N}\}$ of $\mathcal{N}$. Since $a$ lies on the boundary of each orbit in $\mathcal{N}$, every set in $C$ is connected, cf. \cite[Theorem 23.4]{munkres2000topology}. Moreover, for $J,\tilde{J}\in C$ it always holds $J\cap\tilde{J}\supset\{a\}\not=\emptyset$. Thus, by \cite[Theorem 23.3]{munkres2000topology}, $\mathcal{N}$ is connected. From \cite[Proposition 12.25]{sutherland2009introduction} it follows that $\mathcal{N}$ is even path-connected.
\end{proof}
In the proof of \cite[Theorem 4.3]{broughan2003structure} the author does not investigate the case of an isolated equilibrium on the boundary of a node or focus. This is a significant gap, which is closed by the following result.
\begin{proposition}\label{prop:nf_no_isolated_equilibria}
    It holds
    \begin{align*}
        \forall\,\tilde{a}\in\partial\mathcal{N}\cap F^{-1}(\{0\}):\forall\,\rho>0: \left(\mathcal{B}_\rho(\tilde{a})\cap\partial\mathcal{N}\right)\setminus\{\tilde{a}\}\not=\emptyset,
    \end{align*}
    i.e. there are no isolated points with respect to the subspace topology on $\partial\mathcal{N}$. Moreover, for all $\tilde{a}\in\partial\mathcal{N}\cap F^{-1}(\{0\})$ there exists an unbounded orbit $\Gamma\subset\partial\mathcal{N}$ with $\tilde{a}\in w_+(\Gamma)\cup w_-(\Gamma)$, i.e. all equilibria on the boundary of $a$ are attached to an orbit on the boundary of $a$.
\end{proposition}
\begin{proof}
    This proof, carried out by contradiction, relies on constructing a homeomorphism between two appropriately chosen circles without contact, each encircling a focus/node. Since this is highly technical, a very detailed step-by-step proof can be found in the appendix.
\end{proof}
\begin{theorem}\label{thm:nf_unbounded}
    $\partial\mathcal{N}$ is either empty or consists only of equilibria and unbounded orbits. $\mathcal{N}$ is unbounded.
\end{theorem}
\begin{proof}
    If $\partial\mathcal{N}=\emptyset$, then $\mathcal{N}$ is open, closed and nonempty, i.e. $\mathcal{N}=\C{}$ is clearly unbounded. Thus, assume $\partial\mathcal{N}\not=\emptyset$. Additionally, assume w.l.o.g. that $a$ is stable.
    
    Suppose there exists a bounded orbit $\Gamma\subset\partial\mathcal{N}\setminus F^{-1}(\{0\})$. As in the proof of Theorem \ref{thm:center_open}, $K:=\overline{\Gamma}$ is compact and $\Gamma\cup w_\pm(\Gamma)\subset K\subset\partial\mathcal{N}$. By Lemma \ref{lem:no_center_on_boundary_of_invariant_sets}, $\Gamma$ cannot be periodic. Hence, by Theorem \ref{thm:central_tool}, there exist two equilibria $a_+,a_-\in\partial\mathcal{N}\cap F^{-1}(\{0\})$ such that $w_\pm(\Gamma)=\{a_\pm\}$, i.e. $\Gamma$ is either heteroclinic or homoclinic. In both cases we want to derive a contradiction.
    
    Suppose $\Gamma$ is homoclinic, i.e. $a_+=a_-$. By Corollary \ref{cor:globalSector_final_corrollaries_1}, $\Gamma$ is a sector-forming orbit lying in an global elliptic sector $\mathcal{S}_0$ in $a_+$. Choose $z\in\Gamma\subset\mathcal{S}_0\cap\partial\mathcal{N}$. By openness of $\mathcal{S}_0$, cf. Theorem \ref{thm:globalSector_open}, we get $\mathcal{S}_0\cap\mathcal{N}\not=\emptyset$, contradicting the fact that all orbits in $\mathcal{N}$ tend to $a\not\in\{a_+,a_-\}$.
    
    Suppose, $\Gamma$ is heteroclinic, i.e. $a_+\not=a_-$. Since $a\not\in\{a_+,a_-\}$, $a_+$ cannot be a node or focus, since its basin would have nonempty intersection with $\mathcal{N}$.\footnote{If we investigated the case that $a$ is unstable, we would choose $a_-$ instead of $a_+$ at this point.} Additionally, by Lemma \ref{lem:no_center_on_boundary_of_invariant_sets}, $a_+\in\partial\mathcal{N}$ cannot be a center. It follows that $a_+$ must have order $m\ge 2$. By \cite[Proposition 4.3]{kainz2024local}, there exist $\theta\in\mathcal{E}(F,m)$ and $r,\delta>0$ such that $\Gamma$ as well as the orbit $\Gamma(x_0)$ through $x_0\in\{x\in\R{2}:|x-a_+|<r,|\arg(x-a_+)-\theta|<\delta\}=:A$ tend to $a_+$ for $t\to\infty$. We have $\Gamma\cap A\not=\emptyset$. 
    Since $A$ is open, we can choose $x_0\in A\cap\mathcal{N}\not=\emptyset$, i.e. $\Gamma(x_0)\subset\mathcal{N}$ tends to $a_+\not=a$ for $t\to\infty$. This is a contradiction.
    
    We conclude that $\Gamma$ cannot be bounded. All in all, it follows that $\partial\mathcal{N}$ consists only of equilibria (but no centers) and unbounded orbits. Since $\partial\mathcal{N}\not=\emptyset$, we can apply Proposition \ref{prop:nf_no_isolated_equilibria} to conclude that there exists at least one unbounded orbit lying on the boundary of $a$. Thus, also $\mathcal{N}$ is unbounded.
\end{proof}
With the results we have obtained so far, we can now prove that the basin is even simply connected. The crucial factor here is Proposition \ref{prop:nf_no_isolated_equilibria}, which excludes the existence of holes in the form of isolated equilibria on the boundary of $a$.
\begin{theorem}\label{thm:nf_simply_connected}
    $\mathcal{N}$ is simply connected.
\end{theorem}
\begin{proof}
	Let $J\subset\mathcal{N}$ be a loop and suppose that $\Int(J)\setminus\mathcal{N}\not=\emptyset$. Then we can choose $x\in\partial\mathcal{N}\cap\Int(J)
	\not=\emptyset$. If $x$ is not an equilibrium, $\Gamma(x)$ must be unbounded, cf. Theorem \ref{thm:nf_unbounded}. If $x$ is an equilibrium, there exists an orbit $\Gamma\subset\partial\mathcal{N}$ with $x$ in its limit sets, cf. Proposition \ref{prop:nf_no_isolated_equilibria}. In both cases we get an unbounded orbit on the boundary of $a$ having nonempty intersection with $\Int(J)$. Since $\Int(J)$ is bounded, this orbit must intersect $J$. This is a contradiction to the openness of $\mathcal{N}$. Thus, we conclude $\Int(J)\subset\mathcal{N}$ and $J$ can be continuously deformed to a point in $\mathcal{N}$. Since $J$ is arbitrary and $\mathcal{N}$ is already path-connected, it follows that $\mathcal{N}$ is even simply connected.
\end{proof}

To put the basin $\mathcal{N}$ into the framework of the theory of canonical regions and global transversals, we consider the following example of a polynomial vector field.

\begin{example}\label{ex:three_nodes}
     We consider the entire vector field $F:\C{}\to\C{}$ with
    \begin{equation*}
        F(x):=x^3-x=x(x-1)(x+1)\text{.}
    \end{equation*}

    \begin{figure}[ht]
       \centering
        \includegraphics[clip,scale=0.6]{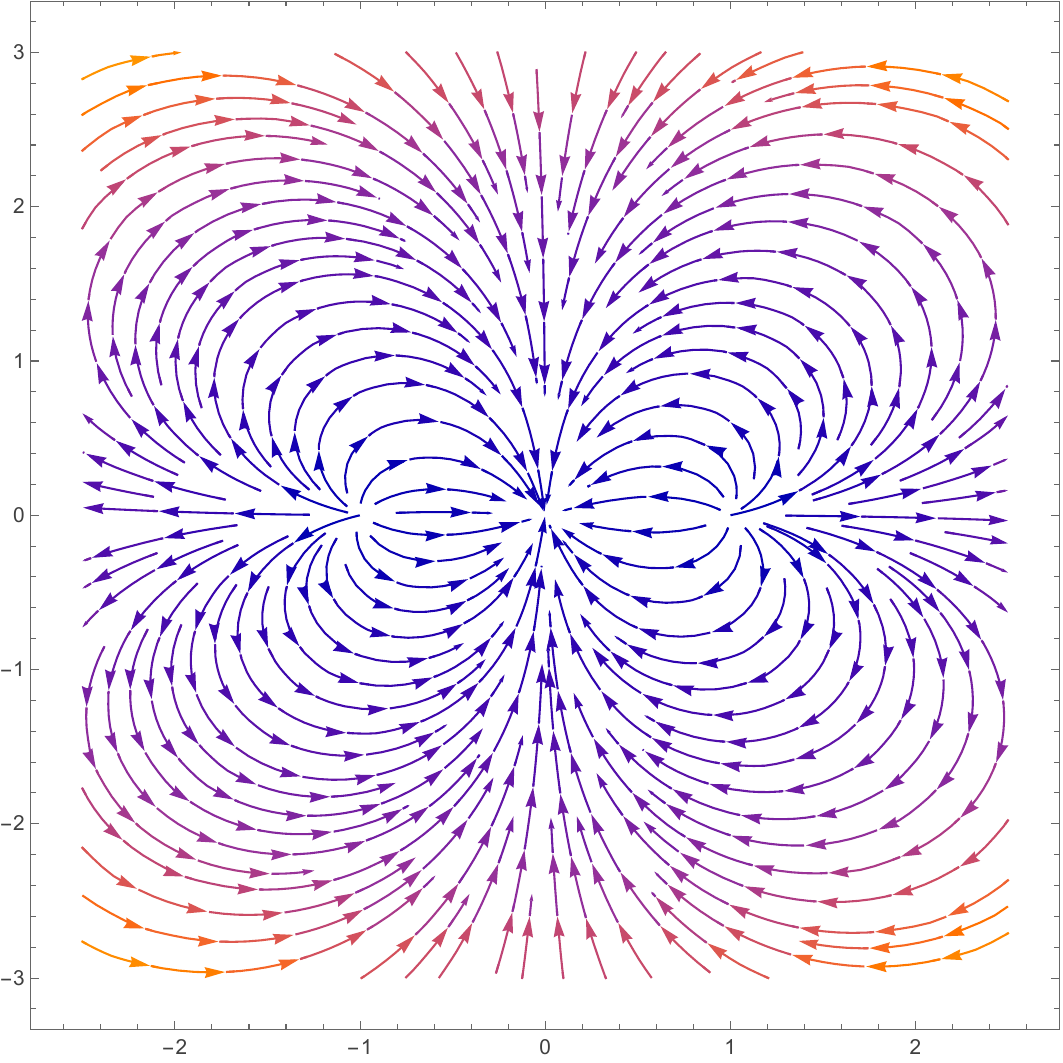}
       \caption{Local phase portrait of system \eqref{eq:planarODE} with $F(x)=x^3-x$, plotted with Wolfram Mathematica.}
   \end{figure}

    All equilibria $F^{-1}(\{0\})=\{0,1,-1\}$ are nodes. The nodes at $\pm 1$ are repelling, while the node at $a=0$ is attracting. By using the notation $\mathrm{i}\R{}:=\{x\in\C:\Re(x)=0\}$, we obtain the following two properties of the phase space by a straightforward calculation:
    \begin{itemize}
        \item $F(\alpha)\in\R{}$ for all $\alpha\in\R{}$, i.e. $\R{}$ is flow-invariant.
        \item $F(\mathrm{i}\beta)\in\mathrm{i}\R{}$ for all $\beta\in\R{}$, i.e. $\mathrm{i}\R{}$ is flow-invariant.
    \end{itemize}
    Thus, $(-1,0)$ as well as $(0,1)$ are orbits with distinct negative limit sets ($\{1\}$ and $\{-1\}$, respectively). Moreover, the upper (denoted $A_1$) and lower (denoted $A_2$) parts of the imaginary axis are orbits, both having empty negative limit sets. All four of these orbits lie in the basin of the node at $a=0$. We therefore conclude that these orbits cannot belong to the same canonical region, cf. \cite[Chapter II.3]{markus1954global}. In particular, since the orbits through arbitrarily close neighborhoods of $A_1$ and $A_2$ have nonempty negative limit sets, $A_1$ and $A_2$ must be separatrices in the sense of \cite{markus1954global}. A similar observation applies to foci if we consider the rotated vector field $\hat{F}(x):=(1+\mathrm{i})F(x)$, $x\in\C$.
\end{example}

Example \ref{ex:three_nodes} demonstrates that the basins of nodes and foci do not necessarily have to form canonical regions. Applying Proposition \ref{prop:nf_no_isolated_equilibria}, we obtain the following result: if $\mathcal{N}$ is a canonical region, then there cannot be any equilibria on $\partial\mathcal{N}$.

In general, nodal and spiral canonical regions can be constructed with the simple linear vector field $x\mapsto\lambda x$ with $\Re(\lambda)\not=0$, cf. \cite[Definition 3.1]{kainz2024local}.

\section{\bf Heteroclinic regions}
In this chapter we consider an entire function $F\in\Hol{}(\C{})$, $F\not\equiv0$, with two equilibria $a_+,a_-\in\C{}$ of \eqref{eq:planarODE} and at least one heteroclinic orbit connecting these two equilibria. By the holomorphy of $F$, the two equilibria have to be nodes, foci or multiple equilibria. Under these assumptions, it is possible to investigate topologically another characteristic set, the so-called "heteroclinic region". Taking up the questions at the end of chapter 3, a heteroclinic region is one possible structure between two adjacent global elliptic sectors. We verify that heteroclinic regions are strip canonical regions in the sense of Markus and Neumann, cf. \cite{markus1954global,neumann1976global}.

\subsection{Geometry of heteroclinic regions}

\begin{definition}
    We define the \textit{heteroclinic region} $\mathcal{H}$ of $F$ between $a_+$ and $a_-$ as
    \begin{align*}
        \mathcal{H}:=\mathcal{H}_+\cap\mathcal{H}_-
    \end{align*}
    with
    \begin{align*}
        \mathcal{H}_{+(-)}:=\left\{x\in\C{}:w_{+(-)}(\Gamma(x))=\{a_{+(-)}\}\right\}\text{.}
    \end{align*}
\end{definition}
\begin{proposition}\label{prop:hetreg_invariant}
    $\mathcal{H}$ and $\partial\mathcal{H}$ are flow-invariant. It holds
    \begin{align*}
		\mathcal{H}=\bigcup_{\mathclap{x\in\mathcal{H}}}\Gamma(x)\text{.}
    \end{align*}
\end{proposition}
\begin{proof}
    This proof is analogous to the proof of Proposition \ref{prop:center_invariant}.
\end{proof}
\begin{theorem}\label{thm:hetreg_open}
    $\mathcal{H}$ is open.
\end{theorem}
\begin{proof}
    Both, $a_+$ as well $a_-$ is either a focus or node or an equilibrium of order $m\in\mathbb{N}\setminus\{1\}$. Let $x\in\mathcal{H}_+$ and assume that $a_+$ is an equilibrium of order $m\ge 2$. By \cite[Proposition 4.3]{kainz2024local}, there exist numbers $r,\theta,\Theta,\tau>0$ such that for all $z\in\{\hat{z}\in\C{}:|\hat{z}-a_+|<r,|\arg(\hat{z}-a_+)-\theta|<\Theta\}=:A$ the orbit through $z$ tends to $a$ for $t\to\infty$ and $M:=\Phi(\tau,x)\in A$. Since $A$ is open, we can choose $\varepsilon>0$ small enough such that $\mathcal{B}_{\varepsilon}(M)\subset A$. By continuity of the flow, cf. \cite[Chapter 2.4, Theorem 4]{perko2013differential}, there exists $\delta>0$ such that $|\Phi(\tau,x)-\Phi(\tau,y)|<\varepsilon$ for all $y\in\mathcal{B}_{\delta}(x)$, i.e. $\mathcal{B}_{\delta}(x)\subset\mathcal{H}_+$ and $\mathcal{H}_+$ is open. Additionally, if $a_+$ is a node or focus, then we can apply Theorem \ref{thm:nf_open} to conclude again the openness of $\mathcal{H}_+$. Since the same argumentation holds for $a_-$, $\mathcal{H}_-$ is also open. It follows that $\mathcal{H}=\mathcal{H}_+\cap\mathcal{H}_-$ is open.
\end{proof}
\begin{proposition}\label{prop:hetreg_equilibria_on_boundary}
    It holds $\partial\mathcal{H}\cap F^{-1}(\{0\})=\{a_+,a_-\}$.
\end{proposition}
\begin{proof}
    Since there exists at least one heteroclinic orbit in $\mathcal{H}$ connecting $a_+$ and $a_-$, we have $\{a_+,a_-\}\subset\partial\mathcal{H}$. Suppose there exists an equilibrium $a\in\partial\mathcal{H}\setminus\{a_+,a_-\}$. By Lemma \ref{lem:no_center_on_boundary_of_invariant_sets}, $a$ cannot be a center. Additionally, since all orbits in $\mathcal{H}$ connect $a_+$ to $a_-$, $a$ cannot be a node a focus. Hence, by using the holomorphy of $F$, we can apply \cite[Theorem 4.4]{kainz2024local} to conclude that $a$ has a FED. In particular, there exists $r>0$ such that $a\in w_+(\Gamma(x))\cup w_-(\Gamma(x))$ for all $x\in\mathcal{B}_r(a)$. We get $\mathcal{B}_r(a)\cap\mathcal{H}\not=\emptyset$, which is also impossible. Thus, we conclude $\partial\mathcal{H}\cap F^{-1}(\{0\})=\{a_+,a_-\}$.
\end{proof}
\begin{theorem}\label{thm:hetreg_unbounded}
    All orbits on $\partial\mathcal{H}\setminus\{a_+,a_-\}$ are unbounded. $\mathcal{H}$ is unbounded.
\end{theorem}
\begin{proof}
    If $\partial\mathcal{H}=\{a_+,a_-\}$, then $\mathcal{H}\not=\emptyset$ is open and closed in the connected set $\C{}\setminus\{a_+,a_-\}$, i.e. $\mathcal{H}=\C{}\setminus\{a_+,a_-\}$ is clearly unbounded. Thus, we assume $\partial\mathcal{H}\setminus\{a_+,a_-\}\not=\emptyset$.
    
    Suppose there exists a bounded orbit $\Gamma\subset\partial\mathcal{H}\setminus\{a_+,a_-\}$. By Lemma \ref{lem:no_center_on_boundary_of_invariant_sets}, $\Gamma$ cannot be periodic. Hence, by Proposition \ref{prop:hetreg_equilibria_on_boundary} and Theorem \ref{thm:central_tool}, we must have $w_+(\Gamma)\cup w_-(\Gamma)\subset\{a_+,a_-\}$. Additionally, by openness of $\mathcal{H}$, $\Gamma$ cannot be heteroclinic. It follows that $\Gamma$ must be homoclinic, w.l.o.g. in $a_+$, i.e. $w_+(\Gamma)=w_-(\Gamma)=\{a_+\}$ and $a_+$ is an equilibrium of order $m\ge 2$. We choose a point $x\in\Gamma\subset\partial\mathcal{H}\setminus\{a_+,a_-\}$. By Lemma \ref{lemma:globalSector_nested_orbits}, there exists $\delta>0$ such that all orbits $\hat{\Gamma}$ through $\mathcal{B}_\delta(x)$ are also homoclinic in $a_+$. Choosing $\hat{\Gamma}\in\mathcal{H}$ leads to a contradiction. We conclude that $\Gamma$ cannot be bounded. Moreover, since $\partial\mathcal{H}\not=\{a_+,a_-\}$, there exists at least one unbounded orbit on $\partial\mathcal{H}$. Therefore, it follows that $\mathcal{H}$ is unbounded.
\end{proof}
\begin{corollary}\label{cor:hetreg_no_cycle_with_heteroclinic_regions}
    \it Let $p\in\mathbb{N}\setminus\{1\}$. It is not possible to find $p$ equilibria $a_1,\ldots,a_p\in F^{-1}(\{0\})$ and $p$ heteroclinic orbits $\Gamma_1,\ldots,\Gamma_p$ such that $\Gamma_p$ is a heteroclinic orbit connecting $a_p$ and $a_1$ and $\Gamma_j$ is a heteroclinic orbit connecting $a_j$ and $a_{j+1}$ for all $j\in\{1,\ldots,p\}$.
\end{corollary}
\begin{proof}
    Suppose we find $p$ equilibria and $p$ heteroclinic orbits with these properties. Then $J:=a_1\cup\ldots\cup a_p\cup\Gamma_1\cup\ldots\cup\Gamma_p$ is a closed piecewise continuously differentiable Jordan curve. Let $\mathcal{H}_j$ be the heteroclinic region defined by $\Gamma_j$. Define the two open sets $U:=\mathcal{H}_1$ and $V:=\mathcal{H}_2\cup\ldots\cup\mathcal{H}_p$, cf. Theorem \ref{thm:hetreg_open}. Then there holds $U\cap V=\emptyset$ as well as, by Theorem \ref{thm:hetreg_open}, $U\cap\Int(J)\not=\emptyset$ and $V\cap\Int(J)\not=\emptyset$. Since $\Int(J)$ is connected, we get $U\cup V\not=\Int(J)$ and thus there exists $x\in\partial U\cap\Int(J)\not=\emptyset$. By Theorem \ref{thm:hetreg_unbounded}, $\Gamma(x)\subset\partial\mathcal{H}_1\cap\Int(J)$ is unbounded, which is a contradiction.
\end{proof}
This illustrates again that holomorphy severely constrains the possible geometric configurations of the phase space. In fact, Corollary \ref{cor:hetreg_no_cycle_with_heteroclinic_regions} does not hold if $F$ is not holomorphic. This happens in particular when one of the equilibria in Corollary \ref{cor:hetreg_no_cycle_with_heteroclinic_regions} is a saddle, cf. e.g. \cite[Chapter 2]{brickman1977conformal}.

Having established all of the above topological properties, we can now assemble them to reveal the full geometric structure of the heteroclinic region. The following argumentation shows the existence of a global transversal for a heteroclinic region. 

\begin{itemize}
	\item[(I)] By using the local FED geometry in a multiple equilibrium, we deduce the existence of a transversal $T$ connecting two adjacent global elliptic sectors that have no common boundary near the multiple equilibrium, cf. \cite[Definition 4.1]{kainz2024local} and Figure \ref{fig:local_elliptic_sector}. Every orbit intersecting $T$, crosses it exactly once. Assume that a heteroclinic region $\mathcal{H}$ lies between these two sectors, and define the set $A=\mathcal{H}\cap T$. Let $q_1,q_2\in A$, and consider the closed Jordan-curve $J$ formed by $\Gamma(q_1)$, $\Gamma(q_2)$, $a_+$ and $a_-$. Connecting $q_1$ and $a_2$ along the segment of $T$, one may ask whether this segment could leave $\mathcal{H}$? The answer is no. In fact, since the segment lies entirely in the interior of $J$, this would contradict Corollary \ref{cor:hetreg_no_cycle_with_heteroclinic_regions} (with $p=2$). It follows that all points in $A$ are connected along the transversal $T$. Taking now the longest possible segment of the transversal $T$ contained in $\mathcal{H}$, we obtain an open smooth arc crossing each closed orbit in $\mathcal{H}$ exactly once. This is illustrated in Figure \ref{fig:hetreg_left}.
	\item[(II)] Let $a$ be a focus or node on the boundary of a heteroclinic region $\mathcal{H}$. By applying the theory of circles without contact around nodes and foci, cf. \cite[\S3, 10.-14., \S7, 1.-2. and \S18, Lemma 3]{andronov1973qualitativetheory}, we find always a continuously differentiable closed path $C\subset\mathcal{H}$ being nowhere tangential to $F$ and satisfying $\Int(C)\cap F^{-1}(\{0\})=\{a\}$. All orbits in the basin of $a$ cross $C$ exactly once. This is illustrated in Figure \ref{fig:hetreg_right}.
	\item[(III)] If $a_+$ or $a_-$ is a multiple equilibrium, we can apply (I) to deduce the existence of global transversal. If both, $a_+$ and $a_-$ are foci/nodes, then we can apply Proposition \ref{prop:nf_no_isolated_equilibria} to show the existence of at least one unbounded orbit $\Gamma$ with $a_+$ in its limit set. Since this orbit intersects $C$ at exactly one point, cf. (II), $C\cap\mathcal{H}$ cannot be a closed segment of $C$, cf. Figure \ref{fig:hetreg_right}. Analogous to the argument in (I), $C\cap\mathcal{H}$ has to be path-connected and is thus a global transversal for $\mathcal{H}$.
\end{itemize}

\begin{minipage}{0.45\textwidth}
    \centering
    \tikzset{every picture/.style={line width=0.75pt}}
    \begin{tikzpicture}[x=0.75pt,y=0.75pt,yscale=-1,xscale=1,scale=0.7]
    
    \draw [color={rgb, 255:red, 0; green, 0; blue, 0 }  ,draw opacity=1 ]   (435.06,307.7) .. controls (422.06,282.7) and (401,243.67) .. (356.71,233.38) ;
    \draw [color={rgb, 255:red, 0; green, 0; blue, 0 }  ,draw opacity=1 ]   (421,322.67) .. controls (411.65,292.9) and (387.98,253.9) .. (356.71,233.38) ;
    \draw [color={rgb, 255:red, 0; green, 0; blue, 0 }  ,draw opacity=1 ]   (377,338.67) .. controls (380,299.67) and (372.02,281.64) .. (356.71,233.38) ;
    \draw [color={rgb, 255:red, 0; green, 0; blue, 0 }  ,draw opacity=1 ]   (327.87,326.87) .. controls (337.07,309.67) and (353.07,270.87) .. (356.71,233.38) ;
    \draw [color={rgb, 255:red, 0; green, 0; blue, 0 }  ,draw opacity=1 ]   (350.67,334.47) .. controls (359.47,305.27) and (362.27,282.87) .. (356.71,233.38) ;
    \draw [color={rgb, 255:red, 0; green, 0; blue, 0 }  ,draw opacity=1 ]   (400.06,337.91) .. controls (398.06,309.91) and (382.31,265.9) .. (356.71,233.38) ;
    \draw [color={rgb, 255:red, 208; green, 2; blue, 27 }  ,draw opacity=1 ][line width=1.5]    (272.89,349.89) .. controls (296.89,331.22) and (335.27,272.07) .. (356.71,233.38) ;
    \draw [color={rgb, 255:red, 74; green, 144; blue, 226 }  ,draw opacity=1 ][line width=1.5]    (318.67,296.47) .. controls (374,324.8) and (413,303.53) .. (442.67,276.87) ;
    \draw [color={rgb, 255:red, 208; green, 2; blue, 27 }  ,draw opacity=1 ][line width=1.5]    (476.06,324.7) .. controls (457.06,287.7) and (407.87,226.47) .. (356.71,233.38) ;
    \draw  [color={rgb, 255:red, 0; green, 220; blue, 20}  ,draw opacity=1 ][fill={rgb, 255:red, 0; green, 220; blue, 20}  ,fill opacity=1 ][line width=2]  (354.73,231.35) .. controls (355.85,230.26) and (357.65,230.28) .. (358.74,231.41) .. controls (359.83,232.53) and (359.81,234.32) .. (358.69,235.41) .. controls (357.57,236.5) and (355.78,236.48) .. (354.68,235.36) .. controls (353.59,234.24) and (353.61,232.45) .. (354.73,231.35) -- cycle;
    \end{tikzpicture}
    \captionsetup[figure]{width=.87\textwidth}
    \captionof{figure}{A complete transversal (blue) for $\mathcal{H}$ near a multiple equilibrium (green). Separatrices are colored red.}
    \label{fig:hetreg_left}
\end{minipage}
\hfill
\begin{minipage}{0.45\textwidth}
    \centering
    \tikzset{every picture/.style={line width=0.75pt}}
    \begin{tikzpicture}[x=0.75pt,y=0.75pt,yscale=-1,xscale=1,scale=0.7]
    
    \draw [color={rgb, 255:red, 0; green, 0; blue, 0 }  ,draw opacity=1 ]   (344.33,237.33) .. controls (327.33,213.33) and (294.33,192.33) .. (250.05,182.05) ;
    \draw [color={rgb, 255:red, 0; green, 0; blue, 0 }  ,draw opacity=1 ]   (314.33,271.33) .. controls (304.33,241.33) and (275.33,203.33) .. (250.05,182.05) ;
    \draw [color={rgb, 255:red, 0; green, 0; blue, 0 }  ,draw opacity=1 ]   (270.33,287.33) .. controls (273.33,248.33) and (265.36,230.3) .. (250.05,182.05) ;
    \draw [color={rgb, 255:red, 0; green, 0; blue, 0 }  ,draw opacity=1 ]   (211.33,263.33) .. controls (230.33,245.33) and (240.33,218.33) .. (250.05,182.05) ;
    \draw [color={rgb, 255:red, 0; green, 0; blue, 0 }  ,draw opacity=1 ]   (164.33,220.33) .. controls (211.33,205.33) and (210.33,199.33) .. (250.05,182.05) ;
    \draw [color={rgb, 255:red, 0; green, 0; blue, 0 }  ,draw opacity=1 ]   (176.33,144.33) .. controls (201.33,161.33) and (204.33,163.33) .. (250.05,182.05) ;
    \draw [color={rgb, 255:red, 0; green, 0; blue, 0 }  ,draw opacity=1 ]   (198.33,121.33) .. controls (233.33,142.33) and (237.31,151.36) .. (250.05,182.05) ;
    \draw [color={rgb, 255:red, 0; green, 0; blue, 0 }  ,draw opacity=1 ]   (246.33,111.33) .. controls (254.33,137.33) and (254.33,144.33) .. (250.05,182.05) ;
    \draw [color={rgb, 255:red, 0; green, 0; blue, 0 }  ,draw opacity=1 ]   (293.33,119.33) .. controls (290.33,141.33) and (279.36,165.36) .. (250.05,182.05) ;
    \draw [color={rgb, 255:red, 0; green, 0; blue, 0 }  ,draw opacity=1 ]   (310.4,132.52) .. controls (299.4,152.52) and (282.4,169.52) .. (250.05,182.05) ;
    \draw [color={rgb, 255:red, 208; green, 2; blue, 27 }  ,draw opacity=1 ][line width=1.5]    (351.83,162.33) .. controls (322.83,172.33) and (298.33,175.33) .. (250.05,182.05) ;
    \draw [color={rgb, 255:red, 0; green, 0; blue, 0 }  ,draw opacity=1 ]   (348.4,208.52) .. controls (328.4,194.52) and (294.4,183.52) .. (250.05,182.05) ;
    \draw [color={rgb, 255:red, 0; green, 0; blue, 0 }  ,draw opacity=1 ]   (167.4,174.52) .. controls (199.4,184.52) and (230.4,184.52) .. (250.05,182.05);
    \draw [color={rgb, 255:red, 74; green, 144; blue, 226 }  ,draw opacity=1 ][line width=1.5]    (301.83,178.33) .. controls (302.83,191.33) and (306.83,213.83) .. (296.83,231.83) .. controls (286.83,249.83) and (269.83,268.83) .. (243.33,260.83) .. controls (216.83,252.83) and (188.83,196.83) .. (202.83,150.83) .. controls (216.83,104.83) and (277.4,108.52) .. (300.4,170.52) ;
    \draw  [color={rgb, 255:red, 0; green, 220; blue, 20}  ,draw opacity=1 ][fill={rgb, 255:red, 0; green, 220; blue, 20}  ,fill opacity=1 ][line width=2]  (248.07,180.02) .. controls (249.19,178.93) and (250.98,178.95) .. (252.07,180.07) .. controls (253.17,181.19) and (253.14,182.99) .. (252.02,184.08) .. controls (250.9,185.17) and (249.11,185.15) .. (248.02,184.03) .. controls (246.92,182.91) and (246.95,181.11) .. (248.07,180.02) -- cycle ;
    \end{tikzpicture}
    \captionsetup[figure]{width=.87\textwidth}
    \captionof{figure}{A complete transversal (blue) for $\mathcal{H}$ near a node (green). The unbounded separatrix is colored red, cf. step (III).}
    \label{fig:hetreg_right}
\end{minipage} 

\begin{theorem}
    $\mathcal{H}$ is connected, path-connected and simply connected. $\mathcal{H}$ is a strip canonical region.
\end{theorem}
\begin{proof}
We use our above considerations in (I)--(III). Let $T:[0,1]\to\mathcal{H}$ be the smooth global transversal through $\mathcal{H}$ and define $Q:=\R{}\times[0,1]\subset\R{2}$. Then the map $\mu:Q\to\mathcal{H}$, $\mu(s,t):=\Phi(t,T(s))$ is a homeomorphism. Moreover, by openness of $\mathcal{H}$, there are no separatrices in $\mathcal{H}$. Since $Q$ is simply connected, this theorem is a direct consequence of \cite[Theorem I]{markus1954globalintegrals} and \cite[Theorem 5.1]{markus1954global}.
\end{proof}

This leads to an interesting observation: global elliptic sectors and heteroclinic regions are both strip regions and thus topologically equivalent in the sense of \cite{markus1954global}. In order to distinguish these two sets we need information about the boundary orbits (separatrices). This underlines that the behavior of separatrices is crucial for determining the global phase portrait of \eqref{eq:planarODE}.

\subsection{Illustrating example}

Several examples of basins, sectors, and regions can be found in \cite{broughan2003holomorphic,broughan2003structure,kainz2024local}. Another example has already been discussed, cf. Example \ref{ex:three_nodes}. The vector field $F(x)=x^3-x$ has two heteroclinic regions: the left and right open half-planes of $\C{}\setminus\{x\in\R{}:|x|>1\}$. These heteroclinic regions share $\mathrm{i}\R{}$ (one equilibrium and two separatrices) as their common boundary. We now turn to a more complicated example to illustrate the theory developed in the preceding chapters.
\begin{example}
    We consider the entire vector field $F:\C{}\to\C{}$ with
    \begin{equation*}
        F(x):=x^2(x-1)(x-\mathrm{i})(x-1-\mathrm{i})\text{.}
    \end{equation*}
    
   \begin{figure}[ht]
       \centering
        \includegraphics[clip,scale=0.75]{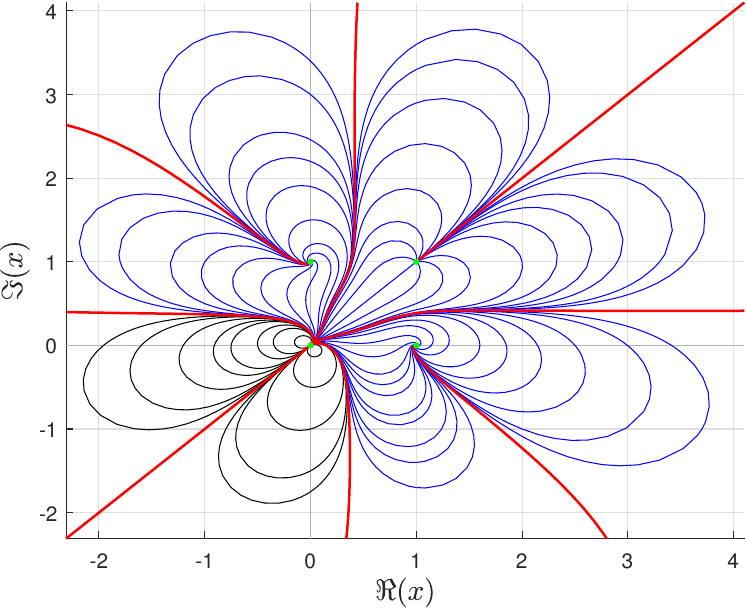}
       \caption{Local phase portrait of system \eqref{eq:planarODE} with $F(x)=x^2(x-1)(x-\mathrm{i})(x-1-\mathrm{i})$, plotted with Matlab.}
   \end{figure}

    We have $F^{-1}(\{0\})=\{0,1,\mathrm{i},1+\mathrm{i}\}$ (colored green). In particular, $a_1=0$ is an equilibrium of order $2$, $a_2=1+\mathrm{i}$ is an attracting node and $a_3=1$ as well as $a_4=\mathrm{i}$ are attracting foci. We observe three heteroclinic regions (colored blue) lying between the two global elliptic sectors: the first between $a_1$ and $a_2$, the second between $a_1$ and $a_3$, and the third between $a_1$ and $a_4$. The red trajectories are the separatrices. In sum, we see five strip canonical regions. We can verify all the topological properties proven in this chapter as well as Corollary \ref{cor:hetreg_no_cycle_with_heteroclinic_regions}.
\end{example}

\newpage

In this context, the question arises as to how many heteroclinic regions can lie between two global elliptic sectors. Finally, we propose the following conjecture, for which a general example is still lacking.

\begin{conj}
    For all $n,m\in\mathbb{N}$ there exists a complex polynomial $F_{n,m}$ of degree $n+m$ such that \eqref{eq:planarODE} has an equilibrium $a$ of order $m$ and $n$ foci or nodes $a_1,\ldots,a_n$ with the following properties:
    \begin{itemize}
        \item[(i)] Either all equilibria $a_1,\ldots,a_n$ are stable or all are unstable.
        \item[(ii)] For all $j\in\{1,\ldots,n\}$ there exists a heteroclinic region between $a$ and $a_j$.
        \item[(iii)] All heteroclinic orbits of \eqref{eq:planarODE} tend to $a$ in the same definite direction.
    \end{itemize}
\end{conj}

\appendix
\section{}

\begin{proof}[Proof of Proposition \ref{prop:globalSector_interior}]
    Let $x\in\mathcal{S}$. Suppose there exists $z_1\in\Int(\Gamma(x)\cup\{a\})$ such that $z_1\not\in\mathcal{S}$. Then there exists $z_2\in\partial\mathcal{S}\cap\Int(\Gamma(x)\cup\{a\})\not=\emptyset$. Since the case $a=z_2$ is not possible, $\Gamma(z_2)$ has to be unbounded, cf. Theorem \ref{thm:globalSector_unbounded}. But $\Gamma(z_2)$ lies in the bounded set $\Int(\Gamma(x)\cup\{a\})$, which is a contradiction. Hence we have $\Int(\Gamma(x)\cup\{a\})\subset\mathcal{S}$. Additionally, by applying Proposition \ref{prop:globalSector_invariant}, we get
    \begin{align*}
        \overline{\Int(\Gamma(x)\cup\{a\})}\setminus\{a\}=\Int(\Gamma(x)\cup\{a\})\cup\Gamma(x)\subset\mathcal{S}\quad\forall\,x\in\mathcal{S}\text{.}
    \end{align*}
    By Theorem \ref{thm:globalSector_open}, for all $x\in\mathcal{S}$ we can choose $y_x\in\Ext(\Gamma(x)\cup\{a\})\cap\mathcal{S}$ such that $\Gamma(x)\subset\Int(\Gamma(y_x)\cup\{a\})$. It follows
     \begin{align*}
        \overline{\Int(\Gamma(x)\cup\{a\})}\setminus\{a\}\subset\Int(\Gamma(y_x)\cup\{a\})\subset\mathcal{S}\quad\forall\,x\in\mathcal{S}\text{.}
    \end{align*}
    These considerations show both equations in (\ref{eq:globalSector_interior}).
    
    Suppose there exist $y_1,y_2\in\mathcal{S}$, $\Gamma(y_1)\not=\Gamma(y_2)$, such that $\Gamma(y_1)$ and $\Gamma(y_2)$ are mutually exterior. 
    
    At first, assume that $\{y_1,y_2\}\cap\Ext(\Xi\cup\{a\})\not=\emptyset$. To fix ideas, assume w.l.o.g. $y_1\in\Ext(\Xi\cup\{a\})$. We have $\Xi\subset\Int(\Gamma(y_1)\cup\{a\})$. Since $\Gamma(y_2)\subset\Ext(\Gamma(y_1)\cup\{a\})$, we get $y_2\in\Ext(\Xi\cup\{a\})$ and thus $\Xi\subset\Int(\Gamma(y_1)\cup\{a\})\subset\Ext(\Gamma(y_2)\cup\{a\})$. It follows the contradiction $y_2\not\in\mathcal{S}$.
    
    Secondly, we assume $\{y_1,y_2\}\subset\Int(\Xi\cup\{a\})$. By using \eqref{eq:globalSector_interior}, we have $\Int(\Gamma(y_1)\cup\{a\})\subset\mathcal{S}$, i.e. all orbits in $\Int(\Gamma(y_1)\cup\{a\})$ are homoclinic in $a$ and $\hat{\Xi}:=\Gamma(y_1)$ is another sector-forming orbit. Define $\hat{\mathcal{S}}:=\mathcal{S}(\hat{\Xi})$. Since $\hat{\Xi}\subset\Int(\Xi\cup\{a\})$, we get $\Xi\subset\hat{\mathcal{S}}$. By using \eqref{eq:globalSector_interior}, we conclude $y_2\in\Int(\Xi\cup\{a\})\subset\hat{\mathcal{S}}$. Additionally, since $y_2\in\Ext(\hat{\Xi}\cup\{a\})$, we must have $\Gamma(y_1)=\hat{\Xi}\subset\Int(\Gamma(y_2)\cup\{a\})$, which is a contradiction.
    
    The last possible case is $\{y_1,y_2\}\cap\Xi\not=\emptyset$. To fix ideas, we assume w.l.o.g. $\Xi=\Gamma(y_1)$. By Definition \ref{def:globalSector}, we must have either $y_2\in\Int(\Xi\cup\{a\})=\Int(\Gamma(y_1)\cup\{a\})$ or $\Gamma(y_1)=\Xi\subset\Int(\Gamma(y_2)\cup\{a\})$. Both results lead to a contradiction.
    %
\end{proof}

\begin{proof}[Proof of Proposition \ref{prop:nf_no_isolated_equilibria}]
	Suppose there exists an isolated point $\tilde{a}\in\partial\mathcal{N}\cap F^{-1}(\{0\})$, i.e. there exists $\rho>0$ such that
	\begin{align*}
		\partial\mathcal{N}\cap(\underbrace{\mathcal{B}_\rho(\tilde{a})\setminus\{\tilde{a}\}}_{=:A})=\left(\mathcal{B}_\rho(\tilde{a})\cap\partial\mathcal{N}\right)\setminus\{\tilde{a}\}=\emptyset\text{.}
	\end{align*}
	It follows\footnote{The simple but technical proof of $\hat{\partial}(\mathcal{N}\cap A)\subset\partial\mathcal{N}\cap A=\emptyset$, where $\hat{\partial}$ denotes the boundary with respect to the subspace topology on $A$, is omitted.} that $\mathcal{N}\cap A\not=\emptyset$ is open and closed with respect to the subspace topology on $A$. Since the punctured circle $A$ is connected, we must have $\mathcal{N}\cap A=A$. Hence, since $\tilde{a}\not\in\mathcal{N}$, we conclude
	\begin{align}\label{eq:nf_isolated1}
		\mathcal{N}\cap\mathcal{B}_\rho(\tilde{a})=\mathcal{B}_\rho(\tilde{a})\setminus\{\tilde{a}\}\text{.}
	\end{align}
	If $\tilde{a}$ was of order $m\ge2$, then, by \cite[Theorem 4.4]{kainz2024local}, there would be at least one homoclinic orbit in $A$ tending to $\tilde{a}$ in both time directions, i.e. equation \eqref{eq:nf_isolated1} would not be fulfilled and $\tilde{a}$ would not be isolated. Thus, $\tilde{a}$ must be a node or focus, cf. Lemma \ref{lem:no_center_on_boundary_of_invariant_sets}. Denote by $\tilde{\mathcal{N}}$ the basin of $F$ in $\tilde{a}$. We reduce $\rho$ such that $\mathcal{B}_\rho(\tilde{a})\subset\tilde{\mathcal{N}}$ and $\mathcal{B}_\rho(a)\subset\mathcal{N}$. By \eqref{eq:nf_isolated1} and Proposition \ref{prop:nf_invariant}, we have $\tilde{\mathcal{N}}\subset\mathcal{N}$.
	
	By applying the theory of circles without contact around nodes and foci, cf. \cite[\S3, 10.-14., \S7, 1.-2. and \S18, Lemma 3]{andronov1973qualitativetheory}, we find two continuously differentiable closed paths $C_1\subset\mathcal{B}_\rho(a)$ and $C_2\subset\mathcal{B}_\rho(\tilde{a})$, $C_1\cap C_2=\emptyset$, being nowhere tangential to $F$ and satisfying $\Int(C_1)\cap F^{-1}(\{0\})=\{a\}$ and $\Int(C_2)\cap F^{-1}(\{0\})=\{\tilde{a}\}$.\footnote{In particular, from the equations (6) and (11) in \cite[\S7, 1.]{andronov1973qualitativetheory} and the remarks made in \cite[\S7, 2.]{andronov1973qualitativetheory} it follows that $C_1$ and $C_2$ can be chosen as linear transformed ellipse.} It holds $\overline{\Int(C_1)}\cap\overline{\Int(C_2)}=\emptyset$. Moreover, every orbit in $\mathcal{N}$ ($\tilde{\mathcal{N}}$) crosses $C_1$ ($C_2$) exactly once, cf. \cite[\S3, 10., Figure 54]{andronov1973qualitativetheory}. We equip $C_1$ and $C_2$ with the subspace topology. Using this construction, the map $\Psi:\mathcal{N}\setminus\{a\}\to\R{}$ is well-defined, when $\Psi(x)$ is the transit time\footnote{cf. \cite[Definition 3.2]{broughan2003structure}.} the point $x\in\mathcal{N}\setminus\{a\}$ needs to $C_1$ along the path $\Gamma(x)$, i.e. $\{\Phi(\Psi(x),x)\}=\Gamma(x)\cap C_1$. Moreover, we define $\eta:C_2\to C_1$ by $\eta(x):=\Phi(\Psi(x),x)$. By the continuous dependence of the flow on initial conditions, cf. \cite[Chapter 2.4, Theorem 4]{perko2013differential}, and the continuous dependence of the transit time integral on the integral limits, $\Psi$ and $\eta$ are both continuous.
	
	One the one hand, since $C_2$ is compact, also $\eta(C_2)$ is compact in $C_1$, cf. \cite[Theorem 26.5]{munkres2000topology}. Thus, $\eta(C_2)$ is closed in $C_1$, cf. \cite[Theorem 26.3]{munkres2000topology}. On the other hand, by Proposition \ref{prop:nf_invariant}, we have $\eta(C_2)\subset\tilde{\mathcal{N}}\cap C_1$. Additionally, since all orbits through $\tilde{\mathcal{N}}\cap C_1\subset\tilde{\mathcal{N}}$ have to cross $C_2$, we also have $\tilde{\mathcal{N}}\cap C_1\subset\eta(C_2)$, i.e. $\eta(C_2)=\tilde{\mathcal{N}}\cap C_1$. Moreover, by Theorem \ref{thm:nf_open}, $\tilde{\mathcal{N}}\cap C_1$ is open in $C_1$. We conclude that $\eta(C_2)\not=\emptyset$ is open and closed in $C_1$. Since $C_1$ is connected, we must have $C_1=\eta(C_2)$, i.e. $\eta$ is surjective. Let $y_1,y_2\in C_2$ be arbitrary. Assume $\eta(y_1)=\eta(y_2)$. Then, by definition of $\eta$, $\Gamma(y_1)=\Gamma(\eta(y_1))=\Gamma(\eta(y_2))=\Gamma(y_2)$. Since all orbits in $\tilde{\mathcal{N}}$ cross $C_2$ exactly once, we must have $y_1=y_2$. Hence, $\eta$ is also injective. It follows that $\eta$ is bijective and thus a homeomorphism, cf. \cite[Theorem 26.6]{munkres2000topology}. We conclude that every orbit in $\mathcal{N}$ and $\tilde{\mathcal{N}}$ crosses $C_1$ as well as $C_2$ exactly once, i.e. we get
	\begin{align}\label{eq:nf_isolated2}
		\hspace{-7mm}\mathcal{H}:=\mathcal{N}\cap \tilde{\mathcal{N}}=\mathcal{N}\setminus\{a\}=\tilde{\mathcal{N}}\setminus\{\tilde{a}\}=\bigcup\limits_{\mathclap{x_2\in C_2}}\Gamma(x_2)=\bigcup\limits_{\mathclap{x_1\in C_1}}\Gamma(x_1)\text{.}
	\end{align}
	Since $C_1\subset\mathcal{H}$ is a Jordan curve, there exists a homeomorphism $h:C_1\to S^1$. Equip the cylinder surface $B:=\R{}\times S^1\subset\R{3}$ with the subspace topology and define the map $\zeta:B\to\mathcal{H}$ by $\zeta(t,x):=\Phi(t,h^{-1}(x))$. By \eqref{eq:nf_isolated2}, we have $\zeta(B)=\mathcal{H}$, i.e. $\zeta$ is surjective. Let $(t_1,x_1),(t_2,x_2)\in B$ be arbitrary and assume $\zeta(t_1,x_1)=\zeta(t_2,x_2)$. Then $\{h^{-1}(x_1),h^{-1}(x_2)\}\subset C_1$ and
	\begin{align*}
		\Gamma(h^{-1}(x_1))=\Gamma(\zeta(t_1,x_1))=\Gamma(\zeta(t_2,x_2))=\Gamma(h^{-1}(x_2))\text{.}
	\end{align*}
	Since all orbits in $\mathcal{H}$ cross $C_1$ exactly once, we get $h^{-1}(x_1)=h^{-1}(x_2)$. Since $h$ is bijective, we conclude $x_1=x_2$. This gives us $\Phi(t_1,h^{-1}(x_1))=\Phi(t_2,h^{-1}(x_1))$ and thus $t_1=t_2$, since $\Gamma(h^{-1}(x_1))$ is not periodic. It follows that $\zeta$ is also injective and thus bijective. Hence, there exists the inverse $\zeta^{-1}$ of $\zeta$, which is given by $\zeta^{-1}(x):=(-\Psi(x),h(\Phi(\Psi(x),x)))$, $x\in\mathcal{H}$. In fact, we easily check that the equation $\zeta(\zeta^{-1}(x))=x$ holds for all $x\in\mathcal{H}$. As a composition of continuous functions, $\zeta$ as well as $\zeta^{-1}$ are continuous and therefore $\zeta$ is a homeomorphism.
	
	By \cite[Corollary 52.5]{munkres2000topology}, the fundamental group of $\mathcal{H}$ is isomorphic to $B$. Additionally, as the circle line $S^1\subset\R{2}$ is a deformation retract of the cylinder surface $B$, we conclude that the fundamental group of $\mathcal{H}$ is isomorphic to the additive group $\mathbb{Z}$, cf. \cite[Theorem 54.5]{munkres2000topology} and \cite[Theorem 58.3]{munkres2000topology}. But $\mathcal{H}$ has the two holes $a$ and $\tilde{a}$, which is impossible for a set having a fundamental group isomorphic to $\mathbb{Z}$. In particular, by using \eqref{eq:nf_isolated2}, we can choose the loop $C_2\cup\Phi([0,\Psi(z_0)],z_0)\cup C_1\subset\mathcal{H}$ with an arbitrary point $z_0\in C_2$. This loop has $a\not\in\mathcal{H}$ as well as $\tilde{a}\not\in\mathcal{H}$ in its interior and thus cannot be transferred to a loop on $S^1$. All in all, we derive a contradiction and $\tilde{a}$ cannot be isolated on $\partial\mathcal{N}$.
	
	Moreover, all equilibria $\hat{a}\in\partial\mathcal{N}\cap F^{-1}(\{0\})$ are nodes, foci or equilibria of order $m\ge2$, cf. Lemma \ref{lem:no_center_on_boundary_of_invariant_sets}. Thus, we can choose $\hat{\rho}$ such that for all $\xi\in\mathcal{B}_{\hat{\rho}}(\hat{a})$ there holds $\hat{a}\in w_+(\Gamma(\xi))\cup w_-(\Gamma(\xi))$, cf. \cite[Definition 3.1]{kainz2024local} and \cite[Theorem 4.4]{kainz2024local}. If we choose $\xi\in(\mathcal{B}_{\hat{\rho}}(\hat{a})\cap\partial\mathcal{N})\setminus\{\hat{a}\}\not=\emptyset$, then $\Gamma:=\Gamma(\xi)\subset\mathcal{N}$ satisfies $\hat{a}\in w_+(\Gamma)\cup w_-(\Gamma)$, i.e. $\hat{a}$ is attached to $\Gamma$.
\end{proof}

\bibliographystyle{plain}

\end{document}